\documentclass[11pt]{amsart}
\usepackage{amssymb,amsmath,mathtools}
\usepackage{enumerate}
\usepackage{xcolor}
\usepackage{caption}
\usepackage[centering]{geometry}

\usepackage{graphicx}
\usepackage{color}
\usepackage{amsmath}
\usepackage{amssymb}
\usepackage{amscd}
\usepackage{framed}
\usepackage{bbm}
\usepackage{tcolorbox}

\usepackage{amsthm}

\newcommand{\R}{\mathbb{R}}
\newcommand{\inr}[1]{\left\langle #1 \right\rangle}

\newcommand{\N}{\mathbb{N}}

\newtheorem{theorem}{Theorem}[section]
\newtheorem{proposition}[theorem]{Proposition}
\newtheorem{lemma}[theorem]{Lemma}
\newtheorem{corollary}[theorem]{Corollary}
\theoremstyle{definition}
\newtheorem{definition}[theorem]{Definition}
\newtheorem{assumption}[theorem]{Assumption}

\newtheorem{remark}[theorem]{Remark}

\newtheorem{question}[theorem]{Question}

\newcommand{\E}{\mathbb{E}}
\renewcommand{\P}{\mathbb{P}}

\newcommand{\V}{\mathrm{\mathbb{V}ar}}

\numberwithin{equation}{section}

\begin{document}

\title{A uniform Dvoretzky-Kiefer-Wolfowitz inequality}

\author{Daniel Bartl}
\address{National University of Singapore, Department of Mathematics \& Department of Statistics and Data Science}
\email{bartld@nus.edu.sg}
\author{Shahar Mendelson}
\address{ETH Z\"urich, Department of Mathematics}
\email{shahar.mendelson@gmail.com}
\date{\today}

\begin{abstract}
	We show that under minimal assumptions on a class of functions $\mathcal{H}$ defined on a probability space $(\mathcal{X},\mu)$, there is a threshold $\Delta_0$ satisfying the following: for every $\Delta\geq\Delta_0$, with probability at least $1-2\exp(-c\Delta m)$ with respect to $\mu^{\otimes m}$,
	 \[ \sup_{h\in\mathcal{H}} \sup_{t\in\mathbb{R}} \left| \mathbb{P}(h(X)\leq t) - \frac{1}{m}\sum_{i=1}^m  1_{(-\infty,t]}(h(X_i))  \right| \leq \sqrt{\Delta};\]
	here $X$ is distributed according to $\mu$ and $(X_i)_{i=1}^m$ are independent copies of $X$.
	
	The value of $\Delta_0$ is determined by an unexpected complexity parameter of the class $\mathcal{H}$ that captures the set's geometry (Talagrand's $\gamma_1$-functional).
	 The bound, the probability estimate and the value of $\Delta_0$ are all  optimal up to a logarithmic factor.
\end{abstract}

\maketitle
\setcounter{equation}{0}
\setcounter{tocdepth}{1}

\section{Introduction}

The Dvoretzky-Kiefer-Wolfowitz inequality is a classical result in probability theory that establishes the speed of convergence of the empirical  distribution function of a random variable to its true counterpart.
More accurately, let $X$ be a real-valued random variable and set $F(t)=\P(X\leq t)$ to be its distribution function.
Let $X_1,\dots,X_m$ be independent copies of $X$, and denote the \emph{empirical distribution function} by
\[F_{m}(t)=\P_m(X\leq t ) = \frac{1}{m}\sum_{i=1}^m 1_{(-\infty,t]}(X_i).\]
The celebrated \emph{Glivenko-Cantelli} theorem \cite{cantelli1933sulla,glivenko1933sulla,vaart1996weak} implies that
\[ \|F_m - F\|_{L_\infty}=\sup_{t\in\R} |F_m(t) - F(t)|
\to 0\]
 almost surely as  $m$ tends to infinity.
The \emph{Dvoretzky-Kiefer-Wolfowitz} (DKW) inequality \cite{dvoretzky1956asymptotic} extends the  Glivenko-Cantelli theorem by specifying the correct speed of convergence.
The following formulation of the DKW inequality  is from \cite{massart1990tight}.

\begin{theorem}
\label{thm:intro.dkw.single}
	For every $\Delta >0$, with probability at least $1-2\exp(-2 \Delta m)$,
	\begin{align} \label{eq:intro.dkw}
	\| F_m  -  F\|_{L_\infty}
	 \leq \sqrt{\Delta}.
\end{align}	
\end{theorem}

The estimate \eqref{eq:intro.dkw} is optimal:
Firstly, the constant factor `2' in the probability estimate $1-2\exp(-2\Delta m)$  cannot be  improved, see \cite{massart1990tight}.
Moreover, if  $X$ is a `non-trivial' random variable\footnote{For instance, if there is no $t\in\R$ for which $\P(X=t)> 0.99$.}, there are  absolute constants $c_1,c_2,c_3$ that satisfy   for every $\Delta\leq \frac{1}{2}$ and $m\geq 4$,
\begin{align*}
 \P\left(	\| F_m  -  F\|_{L_\infty}  \geq c_1\sqrt\Delta\right)
\geq c_2\exp(-c_3\Delta m),
\end{align*}
see, e.g., \cite{bartl2023variance}.
Note that as absolute constants, $c_1,c_2,c_3$ are strictly positive and do not depend on the random variable $X$ (as long as $X$ is `non-trivial' in a reasonable sense).

While Theorem \ref{thm:intro.dkw.single} is a characterisation of the speed of convergence of $\| F_m  -  F\|_{L_\infty}$ in the case of a single random variable, modern applications often require dealing with a  collection of random variables \emph{simultaneously}.
Understanding the uniform behaviour of $\|F_m-F\|_{L_\infty}$  is the goal of this article.

To formulate the question at hand, let $(\mathcal{X},\mu)$ be a probability space and let $X$ be distributed as $\mu$.
Set $X_1,\dots,X_m$ to be independent copies of $X$, and let $\mathcal{H}$ be a class of real valued functions defined on $\mathcal{X}$.
For $h\in \mathcal{H}$, let $F_{h}(t)=\P( h(X)\leq t)$  be the distribution function of  $h(X)$ and put
$$
F_{m,h}(t) = \P_m(h(X)\leq t)=\frac{1}{m}\sum_{i=1}^m 1_{(-\infty,t]}( h(X_i))
$$
to be the corresponding empirical distribution function.

\begin{tcolorbox}
\begin{question}
\label{qu:main-1}
Under what conditions on $X$ and for what values of $\Delta$ does a version of the DKW inequality hold uniformly in $\mathcal{H}$:   that, for some constant $c>0$,  with probability at least $1-2\exp(-c\Delta m)$,
$$
\sup_{h \in \mathcal{H}} \|F_{m,h} - F_{h} \|_{L_\infty} \leq \sqrt{\Delta}?
$$
\end{question}
\end{tcolorbox}

In order to avoid (well understood) measurability issues, we shall assume that $\mathcal{H}$ is regular in an appropriate sense---e.g.\ that there is a countable subset that approximates (pointwise) every function in $\mathcal{H}$.

A particularly important class of examples for $\mathcal{H}$---one that also trivially satisfies the regularity assumption---is a collection of linear functionals that are indexed by a subset of $\mathbb{R}^d$.
For example, consider the case where
 $\mathcal{X}=\R^d$ and $X$ is an isotropic\footnote{That is, $X$ is centred and its covariance is the identity.} random vector. Let  $A\subset S^{d-1}$ and set
\[\mathcal{H}=\mathcal{H}_{{\rm lin}, A}=\{ \inr{\cdot,x} : x \in A\}.\]
Let $F_x(t)=F_{\inr{\cdot, x}}(t)=\P(\inr{X,x}\leq t)$ and  set $F_{m,x}(t)=\P_m(\inr{X,x}\leq t)$.

\subsection{The current state of the art}
Let us begin by describing (known) restrictions on $\Delta$, focusing on classes of the form  $\mathcal{H}_{{\rm lin}, A}$ for various choices of $A$.
The trivial case, $|A|=1$ is fully understood thanks to the optimality of  Theorem \ref{thm:intro.dkw.single}.
In that case,  a nontrivial estimate is possible only when $\Delta\gtrsim \frac{1}{m}$, meaning  that $\Delta\geq c \frac{1}{m}$ for a suitable absolute constant $c>0$.

The other extreme case, $A=S^{d-1}$, was explored in \cite{bartl2022structure}
using Talagrand's concentration inequality and standard methods from VC-theory.
The following formulation is an immediate outcome of   \cite[Corollary 2.8]{bartl2022structure}.

\begin{theorem}[$A=S^{d-1}$]
 \label{thm:intro.vc}
There are absolute constants $c_1$ and $c_2$ such that the following holds.
If  $m\geq d$ and
\begin{equation} \label{eq:intro.rest.Delta.vc}
\Delta \geq c_1 \frac{d}{m},
\end{equation}
then with probability at least $1-2\exp(-c_2\Delta m)$,

\begin{equation}
\label{eq:intro.rest.Delta.vc.estimate}
\sup_{x\in S^{d-1}}
 \left\| F_{m,x}  - F_{x} \right\|_{L_\infty}
\leq   \sqrt{\Delta}.
\end{equation}
\end{theorem}

Theorem \ref{thm:intro.vc} is clearly optimal for $\Delta$ that satisfies  \eqref{eq:intro.rest.Delta.vc}, simply because \eqref{eq:intro.rest.Delta.vc.estimate} is the best one can expect even for a single random variable.
Moreover, one can show (e.g.\ using Proposition \ref{prop:intro.sudakov} below) that the restriction on $\Delta$ in \eqref{eq:intro.rest.Delta.vc} is \emph{optimal}: for example, if  $X$ is the standard gaussian random vector then
\[ \P\left( \sup_{x\in S^{d-1}}
 \left\| F_{m,x}  - F_{x} \right\|_{L_\infty}
\geq c\sqrt \frac{d}{m} \right)\geq 0.9.\]

\vspace{0.5em}
Thus, Question \ref{qu:main-1} is fully understood for classes of linear functionals in the two extreme cases: $|A|=1$ and $A=S^{d-1}$.
At the same time,  the situation for an \emph{arbitrary} set $A\subset S^{d-1}$ is far more subtle.
Indeed, while Theorem \ref{thm:intro.vc}  still applies, the restriction that $\Delta\gtrsim\frac{d}{m}$  (and  $m\gtrsim d$) is independent of $A$ and therefore cannot be optimal.
One would expect a `continuous shift' from $\Delta\sim\frac{d}{m}$ for the entire sphere to $\Delta\sim\frac{1}{m}$ for a singleton---that takes into account some appropriate notion of \emph{complexity} of the set $A$.

There are various notions of complexity that one typically considers (e.g.\ the gaussian mean-width, see \cite{artstein2015asymptotic,talagrand2022upper,vershynin2018high}), but the correct notion of complexity here turns out to be rather surprising: \emph{Talagrand's $\gamma_1$-functional}.
We formulate it already here for general classes $\mathcal{H}$ and set $\|h\|_{L_2} = (\E h^2(X))^{1/2}$.

\begin{definition}
An \emph{admissible sequence} of $\mathcal{H}$ is a collection of subsets $(\mathcal{H}_s)_{s\geq 0}$ of $\mathcal{H}$ that satisfy $|\mathcal{H}_0|= 1$ and  $|\mathcal{H}_s|\leq 2^{2^s}$ for $s\geq 1$.
For $h\in \mathcal{H}$ let $\pi_s h \in \mathcal{H}_s$ be the nearest element to $h$ with respect to $\|\cdot\|_{L_2}$. 
For $\alpha=1,2$ define \emph{Talagrand's $\gamma_\alpha$-functional} by
\[
\gamma_\alpha(\mathcal{H})= \inf \sup_{h \in \mathcal{H}} \sum_{s\geq 0} 2^{s/\alpha} \|h-\pi_s h\|_{L_2},
 \]
 with the infimum taken over all admissible sequences of $\mathcal{H}$.
\end{definition}

\begin{remark}
The definition of the $ \gamma_\alpha$-functional nearly coincides with the one given in Talagrand's book \cite[Definition 2.7.3]{talagrand2022upper}, the only differences are that Talagrand uses partitions of $\mathcal{H}$ instead of finite subsets, and that the partitions are required to be nested.

However, it is standard to verifty that the two languages---of partitions and finite subsets---are essentially interchangeable: one can take arbitrary centres in each member of the partition---resulting in a finite subset, and conversely, one can define a partition of $\mathcal{H}$ via the nearest-point maps $\pi_s$ associated with a finite subset.
Moreover, for any admissible sequence $(\mathcal{H}_s)_{s \geq 0}$, one can construct a nested admissible sequence by defining
\[
\mathcal{H}_0' = \mathcal{H}_0, \quad \mathcal{H}_{s+1}' = \mathcal{H}_0 \cup \mathcal{H}_1 \cup \cdots \cup \mathcal{H}_s.
\]
This modification changes  $\gamma_\alpha$ by at most a multiplicative factor of $2^{1/\alpha}+1$, showing that the two notions of $\gamma_\alpha$ are equivalent.

\end{remark}

\begin{remark}
If $h(X) = \inr{ X, x} $ for some $x \in\mathbb{R}^d$ and $X$ is an isotropic random vector, then $\| h \|_{L_2} = \|x\|_2$.
In that case, we write $\gamma_\alpha(A)$ to denote $\gamma_\alpha(\mathcal{H}_{{\rm lin}, A})$.
\end{remark}

The $\gamma_2$-functional is a natural complexity parameter of $A$: it is equivalent to the set's gaussian mean-width.
Indeed, if $G$ is the standard gaussian random vector in $\R^d$, then by Talagrand's majorizing measures theorem \cite{talagrand1996majorizing,talagrand2022upper}, for any $A\subset\R^d$,
\[ c\, \gamma_2(A)
\leq \E\sup_{x\in A} \inr{G,x}
\leq c'\, \gamma_2(A),\]
for absolute constants $c$ and $c'$.

But despite being a natural candidate, $\gamma_2$ is the wrong choice.
As an indication why, let $(G_i)_{i=1}^m$ be independent copies of $G$, and observe that with high probability,
\begin{align*}
(\ast)=
\sup_{x\in A} \left| \frac{1}{m}\sum_{i=1}^m \inr{G_i,x} - \E \inr{G,x} \right|
\sim \frac{ \gamma_2(A\cup (-A))}{ \sqrt m}.
\end{align*}
Using tail integration it is evident that
\[
(\ast)
= \sup_{x\in A} \left| \int_\R  \big( F_{m,x}(t)  - F_x(t)  \big) \,dt  \right|,\]
which can be significantly smaller than $\sup_{x\in A} \|F_{m,x}-F_x\|_{L_\infty}$.

At the same time, there is no apparent reason why $\gamma_1(A\cup (-A))$ determines  $\sup_{x\in A} \|F_{m,x}-F_x\|_{L_\infty}$.

\begin{remark}
	It worth noting that there are  several situations where $\gamma_1$ and $\gamma_2^2$ are equivalent.
	For instance,  it follows from standard volumetric estimates on covering and packing numbers (see, e.g.\ \cite[Section 2.5]{talagrand2022upper}) that $\gamma_1(S^{d-1})\sim\gamma_2^2(S^{d-1})\sim d$; and if $A$ is the set of $s$-sparse vectors in $S^{d-1}$ then $\gamma_1(A)\sim\gamma_2^2(A)\sim s\log(\frac{ed}{s})$. 	
	However,  in general  there can be a substantial difference between the two.
	For example, consider the spherical cap 
	\[A=\left\{ e_1 +\tfrac{1}{\sqrt d} x : x\in B_2^d\right\}\cap S^{d-1}.\]
 Then clearly $\gamma_\alpha(A) \sim \frac{1}{\sqrt d} \gamma_\alpha(B_2^d)$  for $\alpha=1,2$ and thus $\gamma_1(A)\sim\sqrt{d}$ while $\gamma_2(A)\sim 1$.
\end{remark}

The following answer to Question \ref{qu:main-1} in the case of linear functionals and when the random vector $X$ is the standard gaussian random vector was recently established in \cite{bartl2023empirical}---it is  an immediate outcome of Theorem 1.7 there.

\begin{theorem}[$A\subset S^{d-1}$, $X$=gaussian]
\label{thm:intro.gaussian}
There are absolute constants $c_1$ and $c_2$ for which the following holds.
Let $X=G$ be the standard gaussian vector in $\R^d$. If $A \subset S^{d-1}$ is a symmetric set, $m\geq \gamma_1(A)$, and
\begin{equation} \label{eq:intro.restriction.delta}
\Delta \geq
c_1 \frac{\gamma_1(A)}{m} \log^3\left(\frac{em}{\gamma_1(A)}\right),
\end{equation}
then with probability at least $1-2\exp(-c_2\Delta m)$,
\begin{align}
\label{eq:intro.gaussian.estimate}
\sup_{x\in A}
\left\| F_{m,x} -  F_x \right\|_{L_\infty}
\leq   \sqrt \Delta .
\end{align}
\end{theorem}

Just as in Theorem \ref{thm:intro.vc}, the probability estimate with which \eqref{eq:intro.gaussian.estimate} holds is clearly optimal.
The only question is whether the restriction  on $\Delta$ in \eqref{eq:intro.restriction.delta} can be improved, and it turns out that it cannot (up to, perhaps, logarithmic factors):
There is an absolute constant $c$ such that for \emph{every} $A\subset S^{d-1}$ and all $m\geq m_0(A)$, with probability at least $0.9$,
\begin{equation} \label{eq:lower-1-intro}
\sup_{x\in A}\| F_{m,x} - F_x \|_{L_\infty}
\geq \frac{c}{\sqrt{\log(d)}} \cdot \sqrt\frac{\gamma_1(A)}{m}.
\end{equation}
Thus, ignoring logarithmic factors, for every set $A$, the best possible restriction on $\Delta$ is $\Delta\gtrsim \frac{\gamma_1(A)}{m}$.
More details on \eqref{eq:lower-1-intro} can be found in the Appendix.

\subsection{General classes of functions $\mathcal{H}$}
\label{sec:intro.more.general}

Although Theorem \ref{thm:intro.gaussian} is (almost) a complete answer to Question \ref{qu:main-1} when $\mathcal{H}$ is a class of linear functionals and $X$ is  the standard gaussian random vector, it does not reveal what happens for more general isotropic random vectors, nor what happens when considering an arbitrary class of functions for that matter.
When trying to extend Theorem \ref{thm:intro.gaussian} one immediately encounters a surprise:
The proof of Theorem \ref{thm:intro.gaussian} is based on a chaining argument, and therefore it stands to reason that its assertion should remain true under a suitable tail-assumption on the random vector $X$ (e.g.\ if $\mathcal{H}$ is a class of linear functionals and  $X$ is subgaussian\footnote{An isotropic random vector $X$ is subgaussian with constant $L$ if for every $x\in S^{d-1}$ and $t\geq 0$,  $\P(|\inr{X,x}|\geq Lt )\leq 2\exp(-t^2)$.}).
However, as we show in what follows, a strong tail assumption is not enough.
In fact, it turns out that:

\begin{tcolorbox}
For a uniform DKW inequality to hold, the random variables in the class $\mathcal{H}$ must
\begin{enumerate}[$\cdot$]
\item exhibit a fast tail decay \emph{and}
\item satisfy a suitable small ball condition.
\end{enumerate}
\end{tcolorbox}

Before showing that these assumptions are indeed essential (see Section \ref{sec:ass.needed}), let us formulate our main result, thus  showing that they are sufficient.

\begin{assumption}
\label{ass:tails}
Let $\mathcal{H}$ be a class of functions.
Assume that every function in $\mathcal{H}$ has mean zero and variance 1 and that there are constants $L$  and $D$ for which the following holds:
\begin{enumerate}[(a)]
\item
For every $h,h'\in\mathcal{H}\cup\{0\}$ and $t>0$,
\begin{align}
\label{eq:def.psi1}
 \P\left( |h(X)-h'(X)|\geq  t L \|h-h'\|_{L_2}  \right)
	\leq 2\exp(-t).
\end{align}
\item
For every $h\in\mathcal{H}$, the density $f_h$ of $h(X)$ exists and satisfies $\|f_h\|_{L_\infty}\leq D$.
\end{enumerate}
\end{assumption}


\begin{tcolorbox}
\begin{theorem}
	\label{thm:unif.DKW.class}
	For every  $L,D\geq 1$  there are constants $c_1 $ and $c_2$ that depend only on $L$ and $D$, and for which the following holds.
	If $\mathcal{H}$ satisfies Assumption \ref{ass:tails},  $\gamma_1(\mathcal{H})\geq 1$,  $m\geq \gamma_1(\mathcal{H})$, and
	\[\Delta
	\geq c_1 \frac{\gamma_1(\mathcal{H})}{m} \log^2\left(\frac{em}{\gamma_1(\mathcal{H})}\right),\]
	then with probability at least $1-2\exp(-c_2\Delta m)$,
	\begin{align*}
	\sup_{h\in\mathcal{H}}	\| F_{m,h} -F_{h} \|_{L_\infty}
	\leq  \sqrt{\Delta}.
	\end{align*}
\end{theorem}
\end{tcolorbox}

\begin{remark}
	The assumption  that $\gamma_1(\mathcal{H}) \geq 1$ is made purely to simplify notation.
	It is satisfied trivially when $\mathcal{H}$ is symmetric (i.e., if $h\in\mathcal{H}$ then $-h\in\mathcal{H}$).
	Moreover, if $\gamma_1(\mathcal{H}) < 1$,  one may extend $\mathcal{H}$ to a larger class $\mathcal{H}'$ for which $\gamma_1(\mathcal{H}') = 1$ and then apply Theorem \ref{thm:unif.DKW.class}. 
	In that case, the resulting restriction on $\Delta$ becomes
\[
\Delta \geq \frac{c_1}{m} \log^2(em).
\]
This restriction clearly cannot be improved (at least, up to a logarithmic factor) as it matches the one obtained when $\mathcal{H}$ is a singleton.
\end{remark}

Returning to the class $\mathcal{H}=\mathcal{H}_{{\rm lin}, A}$ of  linear functionals, note that Assumption \ref{ass:tails} is satisfied for any $A\subset S^{d-1}$ if the random vector $X$ is isotropic and for any $x\in S^{d-1}$,
\begin{align}
\label{eq:ass.Rd}
\begin{split}
&\P(|\inr{X,x}|\geq  L t ) \leq 2\exp(-t) \text{ for every }  t\geq 0 \text{ and} \\
&\inr{X,x} \text{ has a density $f_x$ satisfying that $\|f_x\|_{L_\infty}\leq D$}.
\end{split}
\end{align}

In particular, when $\mathcal{H}$ is a class of linear functionals,  the gaussian random vector is a `legal candidate' in Theorem \ref{thm:unif.DKW.class} (with $L$ and $D$ that are absolute constants).
As a result, by the lower bounds in the gaussian case,  the probability estimate is optimal and the restriction on $\Delta$ in Theorem \ref{thm:unif.DKW.class} cannot be improved beyond logarithmic factors.

\vspace{0.5em}
Another important example of a family of random vectors in $\R^d$ that satisfy  \eqref{eq:ass.Rd} is the one endowed by \emph{log-concave} measures.
Indeed, it follows from Borell's lemma (see, e.g., \cite[Lemma 3.5.10]{artstein2015asymptotic}) that  if $X$ is isotropic and log-concave, linear functionals $\inr{X,x}$ exhibit a subexponential tail decay\footnote{An isotropic random vector $X$ has a subexponential tail decay with constant $L$ if for every $x\in S^{d-1}$ and $t\geq 0$, $\P(|\inr{X,x}|\geq L t ) \leq 2\exp(-t)$.} . Also,  the densities of $\inr{X,x}$ are uniformly bounded by an absolute constant \cite[Proposition 2.1]{bobkov2015concentration}.
Therefore, we have the following.

\begin{tcolorbox}
\begin{corollary}
\label{cor:log.cocave}
There are absolute constants $c_1$ and $c_2$ for which the following holds.
Let $X$ be an isotropic, log-concave random vector in $\R^d$.
If  $A \subset S^{d-1}$ is such that $\gamma_1(A)\geq 1$, $m\geq \gamma_1(A)$, and
\[
\Delta \geq
c_1 \frac{\gamma_1(A)}{m} \log^2\left(\frac{em}{\gamma_1(A)}\right),
\]
then with probability at least $1-2\exp(-c_2\Delta m)$,
\begin{equation*}
\sup_{x\in A} \left\| F_{m,x} - F_x\right\| _{L_\infty}
 \leq   \sqrt{\Delta }.
\end{equation*}
\end{corollary}
\end{tcolorbox}

\subsection{On the need for fast tail decay and a small ball condition}
\label{sec:ass.needed}

As it happens, the  necessity of the two  conditions  from  Assumption \ref{ass:tails} is  exhibited by  classes of linear functionals and rather `simple' random vectors: isotropic ones  with iid coordinates.

Let $w$ be a symmetric random variable with variance 1, and consider the random vector $X$ in $\R^d$ whose coordinates are independent copies of $w$.
In particular, $X$ is isotropic.
We also assume that $w$ is `non-trivial', e.g.\ in the sense that $\P(|w|\geq 0.1)\geq 0.1$.

\begin{lemma}
\label{lem:intro.atom}
	Suppose that $w$ has an atom of size $\eta>0$ (i.e.,\ $\P(w=t_0)=\eta$ for some $t_0\in\R$).
	Then for every $m\geq \frac{c_1}{\eta^2}$  there is  $d\in\mathbb{N}$ and a set $A\subset S^{d-1}$ satisfying $\gamma_1(A)\leq 10$ and
	\begin{align}
	\label{eq:intro.atom}
	\P\left( \sup_{x\in A} \| F_{m,x}-F_x\|_{L_\infty} \geq c_2 \cdot \eta \right) \geq 0.9 .
	\end{align}
\end{lemma}

If Theorem \ref{thm:intro.gaussian} were true for the random vector $X=(w_1,\dots,w_d)$, then it would follow that with high (polynomial) probability,
\[
\sup_{x\in A} \left\| F_{m,x} - F_x\right\| _{L_\infty}
\leq  c  \frac{\log(m)}{ \sqrt m },
\]
which is far from the true behaviour as seen in \eqref{eq:intro.atom}.
In particular, the small ball assumption is truly needed.

As an example, let $w$ be the symmetric $\{-1,1\}$-valued random variable, and thus  $X$ is the standard Bernoulli random vector in $\R^d$.
Then $\eta=\frac{1}{2}$, and it follows from Lemma \ref{lem:intro.atom} that the Bernoulli random vector does not satisfy any reasonable uniform version of the DKW inequality.

The obvious conclusion is that once one moves from the gaussian to the subgaussian realm, the best possible bound on
$
\sup_{x\in A} \left\| F_{m,x} - F_x\right\| _{L_\infty}
$
is far weaker than in the gaussian case---unless one imposes more assumptions on $X$---starting with not having significant atoms.

\begin{remark}
\hfill
\begin{enumerate}[(a)]
\item
As it happens, the situation is even worse than what is indicated by Lemma \ref{lem:intro.atom}:
The set $A$ from the lemma can be chosen to have an Euclidean  diameter that is arbitrarily small.
Hence, a satisfactory estimate on $\sup_{x\in A} \left\| F_{m,x} - F_x\right\| _{L_\infty}$ that is based on \emph{any} reasonable complexity parameter of $A$ cannot follow from a tail assumption on the random vector $X$.
\item
It is possible to show that Lemma \ref{lem:intro.atom} remains valid even if $w$ has a density, but that density is unbounded in a suitable sense (roughly put, if there is an interval with length proportional to $\frac{\eta}{m}$ on which the density is of order $m$).
\end{enumerate}
\end{remark}

Following Lemma \ref{lem:intro.atom}, a uniform DKW inequality cannot be true if $w$ has significant atoms or a density that is unbounded in a suitable sense---and that cannot be remedied by a subgaussian tail behaviour.
The next result shows that a tail assumption is also needed:  without  a \emph{subexponential tail decay} a uniform DKW inequality cannot be true.

\begin{lemma}
\label{lem:psi1.needed}
	Let $X$ be an isotropic random vector with independent coordinates distributed according to a symmetric random variable $w$.
	If $w$ does not have a subexponential tail decay, then for every $m\geq 231$ there is  $d\in\mathbb{N}$ and a set $A\subset S^{d-1}$ satisfying $\gamma_1(A)\leq 10$ and
	\[ \P\left( \sup_{x\in A} \| F_{m,x}-F_x\|_{L_\infty} \geq \frac{1}{10}\right) \geq 0.9 . \]	
\end{lemma}

\subsection{Related literature, structure, and notation}

\emph{Related literature:}
As previously indicated, \cite{bartl2023empirical} addresses a similar problem to the one considered here (albeit just for the gaussian random vector).
At the same time, there are substantial differences  between the problem studied in \cite{bartl2023empirical} and Question \ref{qu:main-1}.
These differences are outlined in Section \ref{sec:diff.gaussian} and Section \ref{sec:variance}.
Other than \cite{bartl2023empirical}, the article that seems the most closely related to this one is   \cite{lugosi2020multivariate}, where  an estimate similar to Theorem \ref{thm:unif.DKW.class} is proven, but with a different restriction on $\Delta$.
That restriction is given in terms of a fixed point condition (involving the $L_2(\P)$-covering numbers and the Rademacher complexity associated with the underlying class) and it is far from optimal; for example, one can easily construct sets $A\subset S^{d-1}$ for which  the restriction on $\Delta$ in \cite{lugosi2020multivariate} requires that $\Delta \gtrsim (\frac{\gamma_1(A)}{m} )^{2/5}$.
Other known estimates, such as the ones from  \cite{alexander1987central,alexander1987rates,gine2006concentration,gine2003ratio,koltchinskii2003bounds,koltchinskii2002empirical},
where formulated either in terms of conditions on the random metric structure of the function class (or of sub-level sets of the class), or in terms of the VC-dimension.
All these results are too weak to be applicable in the general case we focus on here.

\emph{Structure:}
In Section \ref{sec:sketch.proof} we present the sketch of the proof of Theorem \ref{thm:unif.DKW.class} and outline the differences between these results and those established in \cite{bartl2023empirical}.
The proof of Theorem \ref{thm:unif.DKW.class} is presented in Section \ref{sec:proof}, and the proofs of Lemma \ref{lem:intro.atom} and Lemma \ref{lem:psi1.needed} can be found in Section \ref{sec:proofs.for.assumptions}. 
In Section \ref{sec:estimating.monotone.functions} we present an application of the uniform DKW inequality  to a collection of estimation problems.
These problems have been studied extensively in recent years,  and their objective is estimating $\E \varphi(\inr{X,x})$ (for various choices of $\varphi$) uniformly in  $x\in A$, using only the random sample $(X_i)_{i=1}^m$ as data. 

\emph{Notation:}
Throughout, $c,c_0,c_1,\ldots$ are absolute constants. Their value may change from line to line.
If a constant $c$ depends on a parameter $B$ we write  $c=c(B)$. If $x \leq cy$ for an absolute constant $c$, that is denoted by $x\lesssim y$, and $x\sim y$ means that $x\lesssim y$ and $x\gtrsim y$. Finally, the cardinality of a finite set $B$ is denoted by $|B|$.

\section{Proof overview}
\label{sec:sketch.proof}

The proof of Theorem~\ref{thm:unif.DKW.class} is based on a chaining argument.
For a comprehensive treatment of chaining methods and their applications, we refer to Talagrand's seminal book \cite{talagrand2022upper}, particularly Chapter 2.4 therein.
Here, we  outline the key steps in the chaining argument used in the proof, which follow the standard structure of chaining-based proofs. 

Let $(\mathcal{H}_s)_{s \geq 0}$ be an (almost) optimal  admissible sequence for $\gamma_1(\mathcal{H})$, meaning that
\begin{align}
\label{eq:optimal.admissible.seq}
\sup_{h\in\mathcal{H}} \sum_{s\geq 0 } 2^{s}  \| h-\pi_s h \|_{L_2}
\leq 2 \gamma_1(\mathcal{H}),
\end{align}
Chaining arguments typically involve breaking the chain to three  regimes. 
To illustrate this, fix $s_1 > s_0 \geq 0$ and decompose $h$ as a telescoping sum,
\[
h = \pi_{s_0} h + \sum_{s=s_0}^{s_1 - 1} (\pi_{s+1} h - \pi_s h) + (h - \pi_{s_1} h).
\]
This decomposition corresponds to the start of the chain ($\pi_{s_0}h$), its middle part (the sum $\sum_{s=s_0}^{s_1 - 1} (\pi_{s+1} h - \pi_s h)$) and its end ($h - \pi_{s_1} h$).
In every chaining argument, the analysis of the three regimes is different, though often the middle part is the heart of the proof.
In our case, due to the (extreme) non-linearity of  $h\mapsto F_h$, the chaining argument requires a somewhat unorthodox analysis, which we outline here; the details are given in Section \ref{sec:proof}.

\vspace{0.5em}
\noindent
\emph{Step 1:} A reduction step (the choice of $s_1$).
First observe that by \eqref{eq:optimal.admissible.seq},
\[\|h-\pi_{s_1} h\|_{L_2}\leq  \frac{2 \gamma_1(\mathcal{H}) }{2^{s_1}};\]
 thus $h$ and $\pi_{s_1} h$ are close if $s_1$ is large enough.
We show in Lemma \ref{lemma:pert-1} that, for any $h\in\mathcal{H}$, with probability 1, 
\begin{align*}
\|F_{m,h} - F_h\|_{L_\infty}
&\leq \|F_{m,\pi_{s_1} h} - F_{\pi_{s_1} h}\|_{L_\infty} + 2\|f_h\|_{L_\infty} \sqrt{\Delta} \\
&\qquad + [\P + \P_m]\left( |\pi_{s_1} h(X) - h(X)| \geq \sqrt{\Delta} \right).
\end{align*}
Moreover,  if $s_1$ is sufficiently large, on a high probability event, 
\[ \sup_{h\in\mathcal{H}} \left| \left\{ i\in\{1,\dots,m\} : |h(X_i)-\pi_{s_1}h(X_i)|\geq\sqrt{\Delta} \right\}\right| \]
is small (of order $\sqrt\Delta m$); in particular the $[\P + \P_m]$-term can be shown to be small (again, of order $\sqrt{\Delta}$) using exponential tail decay---see Section \ref{sec:large.coord} for details.
It therefore remains to estimate $\|F_{m,\pi_{s_1} h} - F_{\pi_{s_1} h}\|_{L_\infty}$.

\vspace{0.5em}
\noindent
\emph{Step 2:}
Consider a grid  in $[0,1]$, 
\begin{align}
\label{eq:def.Udelta}
U_\Delta = \left\{ \ell \Delta : \ell \in \mathbb{N},\, 1 \leq \ell \leq \frac{1}{\Delta}-1  \right\}.
\end{align}
We show in Lemma \ref{lemma:cont1} that
\begin{align*}
\|F_{m,\pi_{s_1} h} - F_{\pi_{s_1} h}\|_{L_\infty} 
&\leq \Delta + \max_{u \in U_\Delta} \left| F_{m,\pi_{s_1} h}(F^{-1}_{\pi_{s_1} h}(u)) - u \right| .
\end{align*}
Observe that for every $u \in U_\Delta$,
\begin{align}
\nonumber
&\left| F_{m,\pi_{s_1} h}(F^{-1}_{\pi_{s_1} h}(u)) - u \right| \\
\nonumber
&\leq \left| F_{m,\pi_{s_1} h}(F^{-1}_{\pi_{s_1} h}(u)) - F_{m,\pi_{s_0} h}(F^{-1}_{\pi_{s_0} h}(u)) \right| + \left| F_{m,\pi_{s_0} h}(F^{-1}_{\pi_{s_0} h}(u)) - u \right| \\
\label{eq:explain.chain}
&\leq \sum_{s=s_0}^{s_1 - 1} \left| F_{m,\pi_{s+1} h}(F^{-1}_{\pi_{s+1} h}(u)) - F_{m,\pi_s h}(F^{-1}_{\pi_s h}(u)) \right| + \|F_{m,\pi_{s_0} h} - F_{\pi_{s_0} h}\|_{L_\infty}.
\end{align}
The term $ \|F_{m,\pi_{s_0} h} - F_{\pi_{s_0} h}\|_{L_\infty}$ corresponds to the start of the chain.
If $s_0$ is chosen appropriately, that term can be controlled  (trivially) using the one-dimensional DKW inequality and the union bound, see Lemma \ref{lem:single.function}.

\vspace{0.5em}
\noindent
\emph{Step 3:}
The heart of the proof is to establish sufficient control on the links in the sum in \eqref{eq:explain.chain}, that is, for all $h, h' \in \mathcal{H}$, estimate
\[ F_{m,h}(F^{-1}_h(u)) - F_{m,h'}(F^{-1}_{h'}(u)) = \frac{1}{m} \sum_{i=1}^m Z_i, \]
where
\[ Z_i = 1_{(-\infty, F^{-1}_h(u)]}(h(X_i)) - 1_{(-\infty, F^{-1}_{h'}(u)]}(h'(X_i)). \]
The random variables $(Z_i)_{i=1}^m$ are independent, centred, satisfy $|Z_i|\leq 1$ almost surely, and have variance $\P(\mathrm{S}^u[h,h'])$ where 
\begin{align}
\label{eq:def.symm.diff}
{\rm S}^u[h, h']
=  \left\{ h(X) \leq F^{-1}_h(u)\right\} \ \bigtriangleup \ \left\{h'(X) \leq F_{h'}^{-1}(u) \right\}
\end{align}
it the \emph{symmetric difference} between the level sets of $h$ and $h’$ at probability level $u \in (0,1)$.
Therefore, by Bernstein's inequality, with probability at least $1-2\exp(-t)$,
\[
\left| F_{m,h}(F^{-1}_h(u)) - F_{m,h'}(F^{-1}_{h'}(u)) \right|
=
\left| \frac{1}{m} \sum_{i=1}^m Z_i \right| \leq 2 \left( \frac{t}{m} +  \cdot \sqrt{ \frac{t \cdot \mathrm{S}^u[h,h']}{m} } \right).
\]
This increment inequality is the core of  the chaining argument.
Thus,  a key step  is to control $\P(\mathrm{S}^u[h,h'])$ in terms of $\|h-h'\|_{L_2}$ (see Lemma \ref{lemma:3}).
The proof is based on using the exponential tail-decay and the fact that the densities are bounded.

\vspace{0.5em}
It is important to stress that the proof here---while seemingly close to the one from \cite{bartl2023empirical} (both are based on a chaining argument; like in most chaining arguments each chain is broken into three parts), the analysis used in the two proofs is actually quite different.
We outline these differences in the next section.

\subsection{Differences between this article and \cite{bartl2023empirical}}
\label{sec:diff.gaussian}

In \cite{bartl2023empirical} we consider $X = G$ a the standard gaussian random vector in $\mathbb{R}^d$, and the class $\mathcal{H} = \mathcal{H}_{{\rm lin}, A}$, consisting of linear functionals $\{\inr{\cdot,x } : x\in A\}$ indexed by $A \subset S^{d-1}$. 
We establish (optimal) estimates on $\sup_{x \in A} |F_{m,x}(t) - F_x(t)|$, uniformly for every $t\in\R$.
In that setting,  a sharper \emph{variance-dependent} estimate (scaling non-uniformly in $t \in \mathbb{R}$) is possible.
However, in Section \ref{sec:variance} we also show that such an estimate is simply false in the  non-gaussian framework we explore here.

The proof in \cite{bartl2023empirical} is also based on a chaining argument and therefore follows the standard pattern of splitting the chain into three regimes.
While the initial part is handled similarly in both cases (as in essentially all chaining arguments)---via a trivial union bound on an individual estimate, the remaining two regimes require distinct analytical treatments.

While there are obvious differences that are due to the slower tail decay---subexponential rather than sub-gaussian, crucially the main  difference lies in the middle regime.

To begin with, the proof in \cite{bartl2023empirical} relies crucially on the rotation invariance of the gaussian distribution.
In particular,  $F_x$ is the same for all $x \in S^{d-1}$, and has an explicit form
\[ F_x(t) = \Phi(t)=\int_{-\infty}^t \frac{1}{\sqrt{2\pi}} e^{-s^2/2} \, ds; \]
that explicit form is important in the argument used in \cite{bartl2023empirical}.

Moreover, the increment condition, based on sufficient control over the corresponding symmetric difference set
\[ \mathrm{S}_t[x, y] = \{ \langle G, x \rangle \leq t \}  \ \bigtriangleup \  \{ \langle G, y \rangle \leq t \}, \quad t\in\R\]
relies heavily on the gaussian nature of the problem.
Specifically,  rotation invariance implies that these sets can be indexed directly by $ t \in \mathbb{R} $, rather than by individual quantile functions $ F_x^{-1}(u) $, $ F_y^{-1}(u) $ as in \eqref{eq:def.symm.diff}.
Moreover,  \( \mathbb{P}( \mathrm{S}_t[x, y] ) \) `scales nicely' with  \( \|x - y\|_2 \) \emph{and}  decays  when $ | t | $ increases.
This  is due to the well-behaved isoperimetric profile of the gaussian and is no longer true in the setup we explore here.

As a result, the core of the chaining argument used here, as well as the technical machinery that is needed are very different from their counterparts in the gaussian case.

\section{Proofs of the  main results}
\label{sec:proof}

\subsection{Preliminary estimates}
\label{sec:prelimiary}

In this section we prove several basic estimates that play a crucial role in the proof of Theorem \ref{thm:unif.DKW.class}.
Through this section, $X$ and $Y$ are real-valued random variables.

The first observation is that  if $t \in \R$ and $r>0$, then
\begin{align} \label{eq:sets1}
 \left\{X \leq t \right\}
 &=  \left\{X \leq t \ \wedge \ |X-Y| < r \right\} \cup \left\{X \leq t \ \wedge \ |X-Y| \geq r
 \right\} \nonumber
\\
& \subset \left\{Y \leq t + r\right\} \cup \left\{ |X-Y| \geq r\right\}.
\end{align}
In particular, using \eqref{eq:sets1} for $\P$ and $\P_m$ respectively,
\begin{align}
\label{eq:P0}
\P(X \leq t)
&\leq \P(Y \leq t+r) + \P(|X-Y| \geq r) \text{ and}\\
 \label{eq:Pm0}
\P_m(X \leq t)
&\leq \P_m(Y \leq t+r) + \P_m(|X-Y| \geq r).
\end{align}

In fact, \eqref{eq:P0} is  two-sided. Indeed,
\begin{align*}
\left\{Y \leq t + r\right\} = & \left\{Y \leq t-r\right\} \cup \left\{ Y \in (t-r,t+r]\right\}
\\
\subset & \left\{X \leq t \right\} \cup \left\{ |X-Y| \geq r\right\} \cup \left\{ Y \in (t-r,t+r]\right\}
\end{align*}
and therefore
\begin{equation} \label{eq:P1}
\P(Y\leq t+r) \leq \P(X\leq t) +\P(|X-Y| \geq r) + \P(Y\in (t-r,t+r]).
\end{equation}
 In particular, it follows from combining \eqref{eq:P0} and \eqref{eq:P1} that if $Y$ has a density $f_Y$ that satisfies  $\|f_Y\|_{L_\infty}\leq D$,  then
\[
\|F_X - F_Y\|_{L_\infty}
\leq   \P(|X-Y| \geq r) + 2Dr,
\]
which is a  `stability condition' on the distribution functions.
%
%
%


Next, let us turn to a simple `empirical  stability' estimate: if $Y$ is a random variable and $\|F_{m,Y}-F_Y\|_{L_\infty}$ is small on some event, then for a random variable $X$ that is `close to $Y$' in an appropriate sense,  $\|F_{m,X}-F_X\|_{L_\infty}$ is also small on the same event. More accurately:

\begin{lemma} \label{lemma:pert-1}
Let $X$ and $Y$ be random variables and assume that $Y$ has a density $f_Y$ which satisfies $\|f_Y\|_{L_\infty} \leq D$.
Then, for every $r>0$,
\begin{align}
\label{eq:pert-1-eq-1}
\|F_{m,X}-F_X\|_{L_\infty}
&\leq \|F_{m,Y}-F_Y\|_{L_\infty}+ [\P+\P_m](|X-Y| \geq r) + 2rD
\end{align}
\end{lemma}

\begin{proof}
Fix $t \in \R$ and $r>0$. 
By \eqref{eq:Pm0},
\begin{align*}
F_{m,X}(t) 
&\leq F_{m,Y}(t+r) + \P_m(|X-Y| \geq r)\\
&\leq  F_Y(t+r) + \|F_{m,Y} - F_Y \|_{L_\infty} + \P_m(|X-Y| \geq r),
\end{align*}
and by \eqref{eq:P1},
$$
F_Y(t+r) \leq F_X(t) + \P(|X-Y| \geq r) + \P(Y\in (t-r,t+r]).
$$

In a similar manner,  by \eqref{eq:P0},
\begin{align*}
F_{X}(t) 
&\leq F_{Y}(t+r) + \P(|X-Y| \geq r) \\
&\leq F_Y(t-r) + \P(Y\in(t-r,t+r])+ \P(|X-Y| \geq r),
\end{align*}
and by \eqref{eq:Pm0} (with the roles of $X$ and $Y$ reversed),
\begin{align*}
F_Y(t-r) 
&\leq F_{m,Y}(t-r) +  \|F_{m,Y} - F_Y \|_{L_\infty} \\
&\leq F_{m,X}(t) + \P_m(|X-Y|\geq r) + \|F_{m,Y} - F_Y \|_{L_\infty}.
\end{align*}
To complete the proof  note that $\P(Y\in (t-r,t+r])\leq 2rD$ since $\|f_Y\|_\infty\leq D$ by the assumption.
\end{proof}

For the rest of this section, we assume  that the random variables $X$ and $Y$ are  absolutely continuous and that their distribution functions are invertible.
In particular, for $u=F_X(t)$ it holds that
$$
\| F_{m,X}-F_X \|_{L_\infty} = \sup_{u\in(0,1)}\left| F_{m,X}(F^{-1}_X(u))-u \right|.
$$
Next, let us turn to a continuity argument of a different flavour: that if $|F_{m,X}(F^{-1}_X(u))-u|$ is small on a `grid' of $ (0,1)$, then $\|F_{m,X} -F_X\|_\infty$ is small as well.
To that end, for $\delta\in(0,1/4)$, set
\[U_\delta=\left\{\delta \ell :  \ell\in \N, \, 1\leq \ell \leq \frac{1-\delta}{\delta} \right\}.\]

\begin{lemma} \label{lemma:cont1}
For every $\delta \in (0,1/4)$,
\begin{equation} \label{eq:cont1.b}
\| F_{m,X} - F_X \|_{L_\infty}
\leq \delta + \max_{u\in U_\delta} \left| F_{m,X}(F^{-1}_X( u))- u \right|  .
\end{equation}
\end{lemma}
\begin{proof}
	Let $t\in\R$, consider $u=F_X(t)$, and set 
	\[\beta=\max_{w\in U_\delta} \left| F_{m,X}(F^{-1}_X( w))- w \right|.\]
	We will only prove one direction of \eqref{eq:cont1.b},  namely that $F_{m,X}(t)- F_X(t)\leq \delta+\beta$.
	The other direction follows the same path and is omitted.
	
	\noindent
\emph{Case 1:} $u\in[0,\delta)$.
	It follows from the monotonicity of $F_X$ that  $t\leq F_X^{-1}(\delta)$, and by  the monotonicity of $F_{m,X}$ and the definition of $\beta$, 
\[ F_{m,X}(t)
\leq F_{m,X}(F^{-1}_X(\delta))
\leq \delta +\beta
\leq F_X(t) +\delta + \beta.
 \]
 
	\noindent
	\emph{Case 2:} $u\in[\delta, 1/2]$.
	Denote by $w$ the smallest element in $U_\delta$ that satisfies $w\geq u$; in particular $|u-w|\leq \delta$.
	By the monotonicity of $F_X$, $F_X^{-1}(u)\leq F_X^{-1}(w)$, and by the monotonicity of $F_{m,X}$,
\[
F_{m,X}(t)
= F_{m,X}(F^{-1}_X(u))
\leq F_{m,X}(F^{-1}_X(w)).\]
	Moreover, 
\begin{equation*}
F_{m,X}(F^{-1}_X(w))
\leq w + \beta
\leq u + \delta + \beta
= F_X(t) + \delta +\beta,
\end{equation*}
and therefore $F_{m,X}(t)-F_{X}(t)\leq \delta + \beta$, as claimed.

 \noindent
 \emph{Case 3:} $u\in(1/2,1]$.
 The proof follows an identical path to the case $u\in[0,1/2]$, and is omitted.
\end{proof}

The final ingredients we need are estimates on the symmetric differences of `level sets' of functions that have the same measure.

\begin{lemma} \label{lemma:3}
Let $X$ and $Y$ be absolutely continuous random variables and assume that their densities satisfy that $\|f_X\|_{L_\infty} ,\|f_Y\|_{L_\infty}\leq D$.
Then for every $u\in(0,1)$ and $r>0$,
$$
\P\left( \left\{ X \leq F^{-1}_X(u)\right\} \ \bigtriangleup \ \left\{Y \leq F_Y^{-1}(u) \right\}\right)
\leq 4\bigl(\P(|X-Y| \geq r) + rD \bigr).
$$
\end{lemma}

\begin{proof}
The proof of Lemma \ref{lemma:3} requires some preparations.
Observe that by the union bound,
\[\P\left( \left\{ X \leq F^{-1}_X(u)\right\} \ \bigtriangleup \ \left\{Y \leq F_Y^{-1}(u) \right\}\right)
\leq 2 \min\{u,1-u\}.\]
Hence, setting $\delta=\P(|X-Y| \geq r)$, we may assume that $\delta<\min\{u,1-u\}$---otherwise the lemma is trivially true.

The key to the proof is the fact that if $\delta<\min\{u,1-u\}$ then
\begin{equation} \label{eq:lemma:2}
F_X^{-1}(u) \leq F_Y^{-1}\left(u+ \delta\right) +r.
\end{equation}

Indeed, following \eqref{eq:P0} once again, we have that $F_Y(t) \leq F_X(t+r) + \delta$ for any $t \in \R$ and $r>0$.
Setting $t= F^{-1}_X(u)-r$,
$$
F_Y(F_X^{-1}(u)-r) \leq u + \delta,
$$
and \eqref{eq:lemma:2} follows by since $u+\delta<1$ by taking $F_Y^{-1}$ on both sides.

We only consider the case $u\leq 1/2$. 
Recall that $\delta=\P(|X-Y| \geq r)$.
If $\delta \geq u$ then the claim is trivially true; thus, assume that  $\delta< u$.

Clearly,
\begin{align*}
 \P \left( X \leq F^{-1}_X(u) \, , \,  Y > F_Y^{-1}(u) \right)
&= \P \left( X \leq F^{-1}_X(u)   \, , \, Y \geq F_Y^{-1}(u+\delta) \right) \\
&\qquad + \P\left( X \leq F^{-1}_X(u) \, , \, Y \in (F^{-1}_Y(u), F^{-1}_Y(u+\delta)) \right)
\\
&=  (1)+(2).
\end{align*}
Firstly, note that
\[(2) \leq \P\left( Y \in (F^{-1}_Y(u) \, , \,F^{-1}_Y(u+\delta)) \right)
=(u+\delta)-u
=\delta.\]
Secondly, by \eqref{eq:lemma:2}, $F_X^{-1}(u) \leq F_Y^{-1}(u+\delta) +r$.
Therefore, if $t=F_X^{-1}(u)$,
$$
(1)
\leq \P \left( X \leq t\, , \,  Y \geq t-r \right).
$$
Also,
\begin{align*}
 \left\{X \leq t \, , \, Y \geq t-r\right\}
& =  \left\{X \leq t \, , \,  Y \geq t- r  \, , \, |X-Y| <  r \right\}
\\
&\qquad  \cup  \left\{X \leq t  \, , \,  Y \geq t-r  \, , \, |X-Y| \geq r \right\}
\\
& \subset  \left\{ X \in [t-2r,t]\right\} \cup \left\{ |X-Y| \geq r\right\},
\end{align*}
and therefore
\begin{align*}
(1)\leq \P \left( X \leq t  \, , \, Y \geq t-r \right)
&\leq  \P\left( X \in [t-2r,t]\right) + \P(|X-Y| \geq r) \\
& \leq  2rD+ \delta.
\end{align*}

An identical argument can be used to estimate $\P ( X > F^{-1}_X(u)  \, , \, Y \leq F_Y^{-1}(u) )$, completing the proof.
\end{proof}

\subsection{Setting up the chaining argument}
From this point on, we are ready for the proof of Theorem \ref{thm:unif.DKW.class}, which will occupy the rest of this section.
We follow the sketch provided in Section \ref{sec:sketch.proof}.

First note that in order to prove Theorem \ref{thm:unif.DKW.class}, it suffices to show that there are constants $c_1',c_2',c_3'$ that depend only on $L$ and $D$ and for which the following holds: if $m\geq\gamma_1(\mathcal{H})\geq 1$ and 
	\begin{align}
	\label{eq:rest.Delta.proof}
	\Delta
	\geq c_1' \frac{\gamma_1(\mathcal{H})}{m} \log^2\left(\frac{em}{\gamma_1(\mathcal{H})}\right),
	\end{align}
	then with probability at least $1-2\exp(-c_2'\Delta m)$, 
	\[	\sup_{h\in\mathcal{H}}	\| F_{m,h} -F_{h} \|_{L_\infty}
	\leq  c_3' \sqrt{\Delta}.\]
	Indeed, if that is established, Theorem \ref{thm:unif.DKW.class} follows by setting  $c_1=c_1' \max\{(c_3')^2,1\}$ and $c_2=c_2'/\max\{(c_3')^2,1\}$.
	
	In particular, by choosing the constants $c_1'$ and $c_3'$ sufficiently large, we may and will assume without loss of generality that $\Delta\leq\frac{1}{10}$ and $m\geq 10$.
	
	Next, note that since $\Delta\mapsto\Delta / \log^2(\frac{e}{\Delta})$ is increasing on $(0,1)$,  the restriction on $\Delta$ in \eqref{eq:rest.Delta.proof} is equivalent to 
\begin{align}
\label{eq:cond.Delta}
c_0 \frac{ \Delta m }{ \log^2(e /\Delta )}
\geq  \gamma_1(\mathcal{H}) 
\end{align}
for some constant $c_0\sim 1/c_1'$.
This formulation is more convenient for our purposes. Accordingly, we fix a sufficiently small constant $c_0$, depending only on $L$, and choose $\Delta$ to satisfy \eqref{eq:cond.Delta}. 
One may take, for instance, $c_0 = c\frac{1}{2L}$ for a small absolute constant $c$ that is specified in Lemma \ref{lem:end.of.chain.dkw}.


To ease notation, we also assume that for every $h\in\mathcal{H}$ the distribution function $F_h$ is invertible; the modifications needed in the general case are straightforward.

\vspace{0.5em}

Further  set $s_0$ and $s_1$ to be the smallest integers that satisfy
\[ 2^{s_0} \geq \Delta m
 \quad \text{and}\quad
2^{s_1} \geq \sqrt{\Delta} m ; \]
in particular $2^{s_0}\leq 2\Delta m$ and $2^{s_1}\leq 2\sqrt{\Delta}m$.

Finally recall that $(\mathcal{H}_s)_{s\geq 0}$ is an almost optimal admissible sequence, see \eqref{eq:optimal.admissible.seq}.

\begin{lemma}
\label{lem:single.function}
	With probability at least $1-2\exp(-\Delta m)$, for every $h\in\mathcal{H}$,
	\begin{align}
	\label{eq:ratio.for.s0}
	\| F_{m,\pi_{s_0}h}-  F_{\pi_{s_0}h} \|_{L_\infty}
	\leq 2\sqrt\Delta.
	\end{align}
\end{lemma}

For completeness, we provide the obvious proof:

\begin{proof}
By Theorem \ref{thm:intro.dkw.single}, \eqref{eq:ratio.for.s0} holds for every fixed $h \in \mathcal{H}$ with probability at least $1-2\exp(-8\Delta m)$. Since
	\[ |\{\pi_{s_0} h : h\in\mathcal{H}\}|
	\leq 2^{ 2^{s_0} }
	\leq \exp( 2\Delta m),\]
the claim follows from the union bound:
\begin{align*}
&\P\left( \sup_{h\in \mathcal{H}} \| F_{m,\pi_{s_0}h}-  F_{\pi_{s_0}h} \|_{L_\infty} > 2\sqrt\Delta \right) \\
&\leq  |\{\pi_{s_0} h : h\in\mathcal{H}\}| \sup_{h\in\mathcal{H}} \P\left(\| F_{m,h}-  F_{h} \|_{L_\infty} > 2 \sqrt\Delta \right) \\
& \leq 2\exp( -\Delta m). \qedhere
\end{align*}
\end{proof}

\subsection{Going further down the chain}
\label{sec:s1}

With $\{\pi_{s_0}h  : h \in\mathcal{H}\}$ being the starting points of each chain, the next step is to go further down the chains --- up to $\pi_{s_1}h$.
The following estimate is at the heart of the proof for Theorem \ref{thm:unif.DKW.class}.

\begin{lemma}
\label{lem:dkw.pi.s1}
	There is an absolute constant $c_1$ and a constant $c_2$ depending only on $L$ and $D$ such that with probability at least $1-2\exp(-c_1\Delta m)$,  for every $h\in\mathcal{H}$,
	\begin{align}
	\label{eq:chaining.first.part}
	 \| F_{m,\pi_{s_1}h} - F_{\pi_{s_1}h} \|_{L_\infty}
	\leq c_2 \sqrt{\Delta}.
	\end{align}
\end{lemma}

The proof of the lemma requires some preparatory steps.
We start with a  (simple yet) crucial observation.

\begin{lemma}
\label{lem:chaning.basic.inequality}
	For every $h,h'\in\mathcal{H}$, $u\in(0,1)$, and $t \geq 0$, with probability at least $1-2\exp(- t)$,
	\begin{align*}
	 \left| F_{m,h} \left( F_{h}^{-1}(u) \right) - F_{m,h'} \left( F_{h'}^{-1}(u) \right) \right|
	\leq 2 \left( \frac{ t }{m} + \sqrt{\frac{ t }{m} \P\left( {\rm S}^u[h,h'] \right) } \right).
	\end{align*}
\end{lemma}
\begin{proof}
For $1\leq i \leq m$, set
\[Z_i=1_{(-\infty, F_{h}^{-1}(u) ]}(h(X_i)) -1_{(-\infty, F_{h'}^{-1}(u) ]}(h'(X_i))\]
and let $Z$ be distributed as $Z_1$.
Using this notation,
\begin{align*}
F_{m,h} \left( F_{h}^{-1}(u) \right) - F_{m,h'} \left( F_{h'}^{-1}(u) \right)
&=\frac{1}{m}\sum_{i=1}^m Z_i,
\end{align*}
and clearly $\|Z\|_{L_\infty}\leq 1$, $\E Z=u-u=0$,  and $\E Z^2=\P({\rm S}^u[h,h'])$.
The claim follows immediately from Bernstein's inequality (see, e.g., \cite[Lemma 4.5.6]{talagrand2022upper}).
\end{proof}

For the next lemma, recall that $U_\Delta =\{ \ell\Delta : \ell\in\N, 1\leq \ell \leq \frac{1}{\Delta}-1\} $, see \eqref{eq:def.Udelta}.

\begin{lemma}
\label{lem:chaning}
 	With probability at least $1-2\exp(-\Delta m)$,  for every $h\in\mathcal{H}$ and $u\in U_\Delta$,
	\begin{align*}
	 &\left| F_{m,\pi_{s_1}h} ( F_{\pi_{s_1}h}^{-1}(u) ) - F_{m,\pi_{s_0}h} ( F_{\pi_{s_0}h}^{-1}(u) ) \right|
	\leq 32 \sup_{h\in\mathcal{H}} \sum_{s=s_0}^{s_1-1} \left( \frac{2^s}{m} + \sqrt{ \frac{2^s}{m} \P\left( {\rm S}^u[\pi_{s+1}h,\pi_s h] \right) } \right).
	\end{align*}
\end{lemma}
\begin{proof}
Let us show that  for every \emph{fixed} $u\in U_\Delta$, with probability at least $1-2\exp(-2 \Delta m)$,
\begin{align}
\label{eq:quation.in.chaining}
\begin{split}
&\sup_{h\in\mathcal{H}}  \left| F_{m,\pi_{s_1}h} ( F_{\pi_{s_1}h}^{-1}(u) ) - F_{m,\pi_{s_0}h} ( F_{\pi_{s_0}h}^{-1}(u) ) \right|\\
&	\leq 32 \sup_{ h\in\mathcal{H} } \sum_{s=s_0}^{s_1-1} \left( \frac{2^s}{m} + \sqrt{ \frac{2^s}{m} \P({\rm S}^u[\pi_{s+1} h,\pi_s h]) } \right).
\end{split}
	\end{align}
Clearly, once \eqref{eq:quation.in.chaining} is established, the wanted estimate  follows from the union bound:
$|U_\Delta|\leq \frac{1}{\Delta}$ and by the restriction on $\Delta$  in \eqref{eq:cond.Delta} we may assume that $\Delta m \geq \log(\frac{e}{\Delta})$.

To prove \eqref{eq:quation.in.chaining}, fix $u\in U_\Delta$ and observe that for every $h\in\mathcal{H}$,
\begin{align*}
&F_{m,\pi_{s_1}h} ( F_{\pi_{s_1}h}^{-1}(u) ) - F_{m,\pi_{s_0}h} ( F_{\pi_{s_0}h}^{-1}(u) )
\\
&=\sum_{s=s_0}^{s_1-1}\left( F_{m,\pi_{s+1}h} ( F_{\pi_{s+1}h}^{-1}(u) ) - F_{m,\pi_{s}h} ( F_{\pi_{s}h}^{-1}(u) )  \right) .
\end{align*}
Lemma \ref{lem:chaning.basic.inequality} implies that with probability at least $1-2\exp(- 2^{s+4})$,
\begin{align}
\label{eq:chaining.1.step}
F_{m,\pi_{s+1}h} ( F_{\pi_{s+1}h}^{-1}(u) ) - F_{m,\pi_{s}h} ( F_{\pi_{s}h}^{-1}(u))
\leq 32 \left( \frac{2^s}{m} + \sqrt{ \frac{2^s}{m} \P({\rm S}^u[\pi_{s+1} h,\pi_sh]) } \right).
\end{align}
And since
\[|\{(\pi_s h,\pi_{s+1} h) : h\in\mathcal{H}\}|
\leq 2^{2^s}\cdot 2^{2^{s+1}}
\leq \exp\left( 2^{s+2} \right),
\]
it follows from the union bound over pairs $(\pi_{s+1}h,\pi_s h)$ that  with probability at least $1-2\exp(-2^{s+3})$, \eqref{eq:chaining.1.step}  holds uniformly for $h\in\mathcal{H}$. Moreover, by the union bound over $s\in\{s_0,\dots,s_1-1\}$,  it is evident that with probability at least
\[
1-\sum_{s=s_0}^{s_1-1} 2\exp(- 2^{s+3})
\geq 1- 2\exp(- 2^{s_0+1})
\geq 1-2\exp(-2\Delta m),
\]
for every $h\in\mathcal{H}$, 
\begin{align*}
 \left| F_{m,\pi_{s_1}h} ( F_{\pi_{s_1}h}^{-1}(u) ) - F_{m,\pi_{s_0}h} ( F_{\pi_{s_0}h}^{-1}(u) ) \right|
\leq  32 \sup_{ h\in\mathcal{H} } \sum_{s=s_0}^{s_1-1} \left( \frac{2^s}{m} + \sqrt{ \frac{2^s}{m} \P({\rm S}^u[\pi_{s+1} h,\pi_s h]) } \right),
\end{align*}
and the claim follows.
\end{proof}

\begin{lemma}
\label{lem:simple.computation}
	For every $0\leq s<s_1$ and $\delta\in(0,2]$,
	\[\frac{2^s}{m}  \delta \log\left(\frac{e}{ \delta  }\right) 
\leq 2\left( \Delta^2 +   \frac{2^s}{m}  \delta  \log\left(\frac{e}{ \Delta }\right) \right)  . \]
\end{lemma}
\begin{proof}
	Suppose first that $\delta>\Delta^2$.
	In this case clearly
	\[ \delta \log\left(\frac{e}{ \delta  }\right)  \leq 2 \delta  \log\left(\frac{e}{ \Delta  }\right) ,\]
	from which the claim follows.
	Next consider the case when $\delta\leq \Delta^2$.
	Since  $2^{s}\leq \sqrt{\Delta}m$ for $s<s_1$,  $ a\mapsto a \log(\frac{e}{a})$ is increasing in $(0,1]$ and  $\sup_{a\in(0,1]} \sqrt{a}\log(\frac{e}{a^2})\leq 2$, 
	\[\frac{2^s}{m}  \cdot  \delta \log\left(\frac{e}{ \delta }\right) 
	\leq \sqrt{\Delta} \cdot \Delta^2 \log\left(\frac{e}{ \Delta^2  }\right) 
	\leq 2 \Delta^2,\]
	and the claim follows.
\end{proof}

\begin{proof}[Proof of Lemma \ref{lem:dkw.pi.s1}]
As explained in Section \ref{sec:sketch.proof}, for every $h\in\mathcal{H}$, we have that 
\begin{align*}
\|F_{m,\pi_{s_1} h} - F_{\pi_{s_1} h}\|_{L_\infty} 
&\leq\max_{u \in U_\Delta} \left| F_{m,\pi_{s_1} h}(F^{-1}_{\pi_{s_1} h}(u)) - F_{m,\pi_{s_0} h}(F^{-1}_{\pi_{s_0} h}(u)) \right| \\
&\quad +  \Delta +  \|F_{m,\pi_{s_0} h} - F_{\pi_{s_0} h}\|_{L_\infty} .
\end{align*}
Moreover, by Lemma \ref{lem:single.function}, with probability at least $1-2\exp(-\Delta m)$, for every $h\in\mathcal{H}$, $\|F_{m,\pi_{s_0} h} - F_{\pi_{s_0} h}\|_{L_\infty} \leq 2 \sqrt{\Delta}$.
Therefore, invoking  Lemma \ref{lem:chaning}, it suffices to  show that for every $u\in U_\Delta$,
\begin{align}
\label{eq:quation.in.chaining.2}
\begin{split}
(\ast)_u
&:=
\sup_{ h\in\mathcal{H} } \sum_{s=s_0}^{s_1-1} \left( \frac{2^s}{m} + \sqrt{ \frac{2^s}{m} \P({\rm S}^u[\pi_{s+1} h,\pi_s h]) } \right)
\\
&\leq c\left(\sqrt\Delta +   \log\left(\frac{e}{\Delta}\right) \sqrt \frac{ \gamma_1(\mathcal{H})}{m} \right),
\end{split}
\end{align}
where $c$ is a constant that depends only on $L$ and $D$.
Indeed, if \eqref{eq:quation.in.chaining.2} is true, then the claim follows from the restriction on $\Delta$ in  \eqref{eq:cond.Delta}.

To estimate $(\ast)_u$, observe that
\[\sum_{s=s_0}^{s_1-1}  \frac{2^s}{m}
\leq \frac{2^{s_1}}{m}
\leq 2\sqrt\Delta,
\]
and all that is left is to control the second term  in the definition of $(\ast)_u$.

To that end,  set
\begin{align*}
\delta_s(h)
&=\|\pi_{s+1}h - \pi_sh\|_{L_2}
\quad\text{and}\\
r_s(h)&=L \delta_s(h) \log\left(\frac{e}{ \delta_s(h)}\right).
\end{align*}
By the subexponential tail decay,
\[\P(| \pi_{s+1}h(X) - \pi_sh(X) |\geq r_s(h) )
\leq 2 \exp\left(-\log\left(\frac{e}{ \delta_s(h) }\right)\right)
\leq \delta_s(h)
\]
and therefore, by Lemma \ref{lemma:3},
\begin{align*}
\P({\rm S}^u[\pi_{s+1} h,\pi_s h])
&\leq  4\big( \P(|\pi_{s+1}h(X) - \pi_sh(X)|\geq r_s(h) ) + r_s(h) D \big) \\
&\leq 4(\delta_s(h) + r_s(h)D) \\
&\leq c_1(L,D) \delta_s(h) \log\left(\frac{e}{ \delta_s(h)}\right).
\end{align*}
It follows that
\begin{align*}
(\ast)_u&\leq c_2 (L,D)\left( \sqrt\Delta  +  \sup_{h\in\mathcal{H}}  \sum_{s=s_0}^{s_1-1} \sqrt{ \frac{2^s}{m}  \delta_s (h) \log\left(\frac{e}{ \delta_s(h)  }\right)  } \right) 
\end{align*}
and thus, the proof is completed if we can show that
\begin{align}
\label{eq:some.formula.1546}
 \sup_{h\in\mathcal{H}}  \sum_{s=s_0}^{s_1-1} \sqrt{ \frac{2^s}{m}  \delta_s (h) \log\left(\frac{e}{ \delta_s(h)  }\right)  } 
 &\leq c_3 \left( \sqrt\Delta   + \log\left(\frac{e}{\Delta  }\right) \sqrt \frac{\gamma_1(\mathcal{H})}{m}  \right).
\end{align}

To that end, fix $h\in\mathcal{H}$.
By an application of the Cauchy-Schwartz inequality followed by Lemma \ref{lem:simple.computation},
\begin{align*}
 \sum_{s=s_0}^{s_1-1} \sqrt{ \frac{2^s}{m}  \delta_s (h) \log\left(\frac{e}{ \delta_s(h)  }\right)  }
&\leq \left( \sum_{s=s_0}^{s_1-1} 1^2 \right)^{1/2} \cdot \left( \sum_{s=s_0}^{s_1-1}  \frac{2^s}{m}  \delta_s (h) \log\left(\frac{e}{ \delta_s(h)  }\right) \right)^{1/2} \\
&\leq \left( s_1-s_0 \right)^{1/2}  \cdot  \left( 2\sum_{s=s_0}^{s_1-1}   \left( \Delta^2 +  \frac{2^s}{m}  \delta_s (h)\log\left(\frac{e}{\Delta}\right) \right)   \right)^{1/2} \\
&=:(\ast\ast).
\end{align*}
Using the subadditivity of $a\mapsto \sqrt{a}$ and the choice of $(\mathcal{H}_s)_{s\geq 0}$ as an almost optimal admissible sequence (see \eqref{eq:optimal.admissible.seq}),
\[ (\ast\ast)
\leq (s_1-s_0)^{1/2} \cdot \left(  (2(s_1-s_0))^{1/2} \Delta + 2 \left(\frac{\gamma_1(\mathcal{H})}{m} \log\left(\frac{e}{\Delta}\right)\right)^{1/2} \right) .\]
Since  $2^{s_0}\geq \Delta m$ and  $2^{s_1}\leq 2\sqrt{\Delta} m$,  it follows that
\[ s_1-s_0 \leq \log_2 (2\sqrt\Delta m) - \log_2(\Delta m)
\leq \log\left(\frac{4}{\Delta} \right)\]
and thus
\begin{align*}
(\ast\ast) 
&\leq  2\log\left(\frac{4}{\Delta} \right)  \Delta + 2 \log\left(\frac{4}{\Delta}  \right) \sqrt\frac{\gamma_1(\mathcal{H})}{m},
\end{align*}
from which \eqref{eq:some.formula.1546} follows by noting that  $\log(\frac{4}{\Delta} )\Delta \leq  2 \sqrt{\Delta}$.
\end{proof}

\subsection{Controlling $[\P+\P_m](|\pi_{s_1}h(X)- h(X)|\geq \sqrt\Delta)$}
\label{sec:large.coord}

The next component in the proof of Theorem \ref{thm:unif.DKW.class} is that with high probability,
\[\sup_{h \in \mathcal{H}} [\P+\P_m]\left(|h(X)-\pi_{s_1}h(X)| \geq \sqrt\Delta \right)
\leq c\sqrt\Delta. \]
The estimate on the $\P$-probability is straightforward:
Because $(\mathcal{H}_s)_{s\geq 0}$ is an almost optimal admissible sequence  for $\gamma_1(\mathcal{H})$  (see \eqref{eq:optimal.admissible.seq}) and  $2^{s_1}\geq \sqrt{\Delta} m$, 
\begin{align}
\label{eq:L2.diff.h.minus.pish}
\|h-\pi_{s_1}h\|_{L_2}
\leq 2\frac{\gamma_1(\mathcal{H})}{2^{s_1}}
\leq 2\frac{\gamma_1(\mathcal{H})}{\sqrt\Delta m}
\leq \frac{ \sqrt \Delta}{L \log^2 (e/\Delta)},
\end{align}
where the last inequality follows from the restriction on $\Delta$ in \eqref{eq:cond.Delta} and using that  the constant therein is chosen to satisfy $c_0\leq \frac{1}{2L}$.
Thus, invoking the subexponential tail decay,
\begin{align}
\label{eq:actual.tails}
\P \left( | h(X) - \pi_{s_1}h(X) | \geq \sqrt{\Delta} \right)
\leq 2\exp\left(- \log^2\left( \frac{e}{\Delta} \right)\right)
\leq \sqrt\Delta,
\end{align}
as required.

Obtaining a similar estimate for the empirical measure $\P_m$ is more subtle.

\begin{lemma}
\label{lem:end.of.chain.dkw}
	There is an absolute constant $c_1$ satisfying that with  probability at least $1-2\exp(-c_1\Delta m)$, for every $h\in\mathcal{H}$,
\begin{align}
\label{eq:end.chain}
\P_m \left( | h(X) - \pi_{s_1}h(X) | \geq \sqrt{\Delta} \right)
\leq 2 \sqrt\Delta.
\end{align}
\end{lemma}

Note that \eqref{eq:end.chain} is  equivalent to
\begin{align}
\label{eq:end.of.chaing.big.coordinates}
\left| \left\{ i  : |h(X_i)-\pi_{s_1}h(X_i)| \geq \sqrt{\Delta} \right\} \right|
\leq  2\sqrt{\Delta }m,
\end{align}
a fact that will be used in the proof of Lemma \ref{lem:end.of.chain.dkw}:
Denote by $(a^\ast_i)_{i=1}^m$ the monotone non-increasing rearrangement of $(|a_i|)_{i=1}^m$.
Setting $a=(h(X_i)-\pi_{s_1}h(X_i))_{i=1}^m$,  \eqref{eq:end.of.chaing.big.coordinates} holds if $a^\ast_{ 2\sqrt\Delta m}<  \sqrt{\Delta}$; in particular it suffices that $a^\ast_{ 2^{s_1}}<  \sqrt{\Delta}$.
Clearly
\[
a_{k}^\ast \leq \frac{1}{k} \sum_{i=1}^k a_i^\ast
\]
for every $1\leq k\leq m$, and Lemma \ref{lem:end.of.chain.dkw} follows once we show that
with probability at least $1-2\exp(-c_1 \Delta m)$,
\begin{align}
\label{eq:end.of.chaing.big.coordinates.2}
\sup_{h\in\mathcal{H}}  \frac{1}{2^{s_1}}  \sum_{i=1}^{2^{s_1}} \left( h(X_i)-\pi_{s_1}h(X_i) \right)^\ast
<  \sqrt\Delta.
\end{align}

\vspace{0.5em}
The proof of \eqref{eq:end.of.chaing.big.coordinates.2} is based on a (rather standard) chaining argument.
For $1\leq k \leq m$, set
\[\mathcal{S}_k =\{ b\in\{-1,0,1\}^m : |\{i : b_i\neq 0\}|= k\},\]
put $\mathcal{S}_k= \mathcal{S}_m=\{-1,1\}^m$ for $k>m$,
and note that $\sum_{i=1}^k a_i^\ast=\max_{b\in \mathcal{S}_k}   \sum_{i=1}^m b_i a_i  $.
The next observation is an immediate consequence of the  $\psi_1$-version of Bernstein's inequality and its proof is presented for the sake of completeness.

\begin{lemma}
\label{lem:alpt.our.setting}
	There exists an absolute constant  $c_1$ such that the following holds.
	For every $k \geq 1$, $b\in\mathcal{S}_k$,  $h,h'\in\mathcal{H}$ and   $t\geq 0$, with probability at least $1-2\exp(- t)$,
	\[ \left| \sum_{i=1}^m  b_i \left(h(X_i)- h'(X_i) \right)   \right|
	\leq c_1 L \left( \sqrt{ t k }  +  t \right) \|h -h' \|_{L_2}. \]
\end{lemma}
\begin{proof}
	Assume first that $k\leq m$ and let $a\in\R^m$ be a rearrangement of $b$ that is supported on the  first $k$ coordinates.
	Set  $Z_i=h(X_i)-h'(X_i)$, and clearly  $\sum_{i=1}^m b_i Z_i$ and $\sum_{i=1}^k  a_i Z_i$ have the same distribution since $(Z_i)_{i=1}^m$ is iid.
	The wanted estimate follows from a version of Bernstein's inequality, see, e.g.\  \cite[Lemma 3.6]{adamczak2011restricted}.
	The case $k>m$ follows trivially from the corresponding estimate for $k=m$.
\end{proof}

\begin{proof}[Proof of Lemma \ref{lem:end.of.chain.dkw}]
	Recall that $ 2^{s_1}\sim  \sqrt\Delta m$.
	Therefore,
	\[ |\mathcal{S}_{2^{s_1}}|
	= 2^{2^{s_1}} \binom{m}{2^{s_1}}
	\leq \exp\left(c_1 \sqrt\Delta m\log\left(\frac{e}{ \Delta }\right)\right),
 \]
	and by the union bound it suffices to show that  for every \emph{fixed} $b\in\mathcal{S}_{2^{s_1}}$, with probability at least  $1-2\exp(-2c_1 \sqrt \Delta m \log(\frac{e}{\Delta}))$,
\[ \sup_{h\in\mathcal{H}}  \frac{1}{2^{s_1} }  \sum_{i=1}^m b_i \left( h(X_i)-\pi_{s_1}h(X_i) \right)
< \sqrt\Delta.\]

To that end, fix  $b\in \mathcal{S}_{2^{s_1}}$ and let $\eta=\eta(c_1)\geq1$ be an absolute constant to be specified in what follows.
For $s\geq s_1$,  Lemma \ref{lem:alpt.our.setting} (with the choices of $k=2^s$ and $t= \eta 2^s\log(\frac{e}{\Delta})$) shows that for every $h\in\mathcal{H}$,  with probability at least $1-2\exp(-\eta 2^s\log(\frac{e}{\Delta}))$,
\begin{align}
\label{eq:large.coord.lemma}
 \left| \sum_{i=1}^m b_i ( \pi_{s+1}h(X_i)-\pi_{s} h(X_i) ) \right|
\leq c_2 L \eta 2^s\log\left(\frac{e}{\Delta}\right) \| \pi_{s+1}h-\pi_{s} h \|_{L_2}.
\end{align}
Recalling that $|\{ ( \pi_{s+1}h, \pi_s h ): h\in\mathcal{H} \} |
\leq \exp\left (2^{s+2} \right)$,
it follows from the union bound over $s\geq s_1$ and pairs $(\pi_{s+1}h,\pi_s h)$ that \eqref{eq:large.coord.lemma} holds uniformly in $h\in\mathcal{H}$
with probability at least
\[
1- \sum_{s\geq s_1} 2\exp\left( -\eta 2^s\log\left(\frac{e}{\Delta}\right)\right)
\geq 1-2\exp\left( -2c_1 \sqrt\Delta m\log\left( \frac{e}{\Delta } \right) \right),
\]
where the last inequality holds for a suitable choice of $\eta=\eta(c_1)$.
On that event,
\begin{align*}
&\sup_{h\in\mathcal{H}}  \frac{1}{2^{s_1} }  \sum_{i=1}^m b_i \left( h(X_i)-\pi_{s_1}h(X_i) \right) \\
&= \sup_{h\in\mathcal{H}}  \frac{1}{2^{s_1}} \sum_{s\geq s_1}   \sum_{i=1}^m b_i ( \pi_{s+1}h(X_i)-\pi_{s} h(X_i) ) \\
&\leq  c_2 L \eta   \log\left(\frac{e}{\Delta}\right)  \sup_{h\in\mathcal{H}}  \frac{1}{2^{s_1}} \sum_{s\geq s_1} 2^s \|\pi_{s+1}h -\pi_s h\|_{L_2} \\
&\leq  c_2 L\eta   \log\left(\frac{e}{\Delta}\right) \frac{2\gamma_1(\mathcal{H})}{\sqrt \Delta m}
=(\ast),
\end{align*}
where the last inequality holds since    $(\mathcal{H}_s)_{s\geq 0}$ is an almost optimal admissible sequence (see \eqref{eq:optimal.admissible.seq}) and because $2^{s_1}\geq \sqrt{\Delta } m$.
Finally, by the restriction on $\Delta$ in \eqref{eq:cond.Delta} (assuming that the constant $c_0$ therein satisfies $c_0< \frac{1}{2c_2 \eta L }$), we have that $(\ast)<\sqrt{\Delta}$.
\end{proof}

\subsection{Putting everything together---proof of Theorem \ref{thm:unif.DKW.class}}
\label{sec:putting.together}

Let $\Omega(\mathbb{X})$ be the event in which the assertions of  Lemma \ref{lem:dkw.pi.s1} and Lemma \ref{lem:end.of.chain.dkw} hold; in particular
$\P(\Omega(\mathbb{X})) \geq 1-2\exp(-c_1\Delta m)$.
By the estimate in \eqref{eq:actual.tails}, it follows that for every realization $(X_i)_{i=1}^m\in\Omega(\mathbb{X})$ and $h\in\mathcal{H}$,
\begin{equation*}
\|F_{ m, \pi_{s_1} h }-F_{ \pi_{s_1} h }\|_{L_\infty}
\leq c_2(L,D) \sqrt\Delta ,
\end{equation*}
and
\begin{equation*}
[\P+\P_m]\left( |\pi_{s_1} h(X) - h(X) |\geq \sqrt\Delta \right)
\leq 3\sqrt \Delta.
\end{equation*}
Therefore, an application of  Lemma \ref{lemma:pert-1} (with $r=\sqrt\Delta$) shows that
\begin{align*}
\| F_{m,h}-F_h \|_{L_\infty}
&\leq  \|F_{ m, \pi_{s_1} h } -F_{ \pi_{s_1} h } \|_{L_\infty}  + [\P+\P_m]\left( |h(X) - \pi_{s_1} h(X) |\geq \sqrt\Delta \right)  + 2D\sqrt \Delta\\
&\leq (c_2(L,D)+3+2D)\sqrt{\Delta},
\end{align*}
as claimed.
\qed

\section{On the assumptions}
\label{sec:proofs.for.assumptions}

In this section we provide the proofs for the claims made in  Section \ref{sec:ass.needed} that a fast tail decay and a small ball condition are necessary  assumptions for a uniform DKW inequality to hold true.
In addition, we show that a variance dependent DKW inequality as established in \cite{bartl2023empirical} for the gaussian measure does not hold true in the present setting.

Recall that $w$ is a symmetric random variable with variance 1 and that  $X$ has independent coordinates distributed according to $w$.
We begin with the proof of Lemma \ref{lem:psi1.needed}, which is the simpler of the two.

\subsection{Proof of Lemma \ref{lem:psi1.needed}}

	Let $\delta\in(0,\frac{1}{10})$, set $d\in\mathbb{N}$ to be specified in what follows, and put
	\begin{align}
	\label{eq:def.set.lower.bound}
	A=\left\{ \sqrt{1-\delta^2}  \cdot  e_1 + \delta \cdot  e_k : 2\leq k \leq d \right\} \subset S^{d-1}.
	\end{align}
	Note that $\gamma_1(A)\leq c_1 \delta \log(d)$.
	Indeed, this follows by considering the admissible sequence $(A_s)_{s\geq 0}$ defined by $A_s=\{ \sqrt{1-\delta^2}    e_1 + \delta   e_2 \}$ when $2^{2^s}<d$ and $A_s=A$ otherwise, and noting that  ${\rm diam}(A,\|\cdot\|_2)\leq 2\delta$.
	
	By  Markov's inequality  $\P(|\inr{X,x}|\geq 2) \leq \frac{1}{4} \E \inr{X,x}^2 =\frac{1}{4}$ and thus, since $\inr{X,x}$ is a symmetric random variable,
	\[
	\sup_{x\in A} F_{x}(-2)\leq \frac{1}{8}.
	\]
	
	Next, let us show that  with  probability at least $0.9$, there exists $x\in A$ which satisfies that  $F_{m,x}(-2)\geq\frac{4}{10}$; in particular
	\[ \P\left(\sup_{x\in A} |F_{m,x}(-2)-F_x(-2)|\geq\frac{1}{10}\right)\geq 0.9,\]
	as claimed.
	
	To that end, set
	\[ I= \{ i\in\{1,\dots,m\} : \inr{X_i,e_1}\leq 0\}\]
	and denote by $\Omega_1(\mathbb{X})$ the event in which $|I|\geq \frac{4}{10} m$.
	Since $\P(\inr{X,e_1}\leq 0)=\frac{1}{2}$, it follows e.g.\ from the Chernoff–Hoeffding inequality that 
	\[\P(\Omega_1(\mathbb{X}))
	=\P\left( {\rm Binomial}\left(m,\frac{1}{2} \right)  \geq \frac{4}{10}m \right)
	\geq 1- \exp\left(\frac{-m}{50}\right)
	\geq 0.99\]
	where the last inequality holds for $m\geq 231$.
	Consider
	\[ \Omega_2(\mathbb{X})=\left\{ \text{there is }2\leq k \leq d \text{ such that} \inr{X_i,e_k}\leq -\frac{2}{\delta} \text{ for every  } 1\leq i \leq m\right\}\]
	and set
	\[\beta
	=\P\left( w\leq -\frac{2}{\delta}\right).\]
	Therefore,
	\begin{align*}
	\P(\Omega_2(\mathbb{X}))
	=1-(1-\beta^m)^{d-1}
	&=1-\exp\big((d-1)\log(1-\beta^m)\big) \\
	&\geq 1- \exp\left( -(d-1)\beta^m) \right),
	\end{align*}
	where we used that $\log(1-\lambda)\leq  -\lambda$ for $\lambda\in(0,1)$.
	Setting $d= c_2\beta^{-m}$, it follows that $\P(\Omega_2(\mathbb{X}))\geq 0.99$.
	
	Clearly, for every realization $(X_i)_{i=1}^m\in\Omega_1(\mathbb{X})\cap\Omega_2(\mathbb{X})$ there exists $x\in A$ for which
	\[ \inr{X_i,x}
	= \sqrt{1-\delta^2}\inr{X_i,e_1}  + \delta \inr{X_i,e_k}
	\leq - 2 \quad\text{for every } i\in I,
\]
	and thus $F_{m,x}(-2)\geq \frac{|I|}{m}\geq \frac{4}{10}$.

	It remains to show that there is some  $\delta< \frac{1}{10}$ for which  $\gamma_1(A)\leq 10$.
	Since $d=c_2\beta^{-m}$, we have that $\gamma_1(A)\leq c_3\delta m \log(\frac{1}{\beta})$; therefore $\gamma_1(A)\leq 10$ if
	\begin{align*}
	\beta
	=\P\left( w\leq -\frac{2}{\delta} \right)
	\geq \exp\left( - \frac{ 10}{c_3\delta m} \right).
	\end{align*}
Finally, $w$ does not have a subexponential tail decay, and in particular, for every $L\geq 1$ there exists $t>20$ that satisfies $\P(w\leq -t) \geq \exp(-t/L)$.
	The claim  follows by setting $\delta=\frac{2}{t}$ for $t$ corresponding to $L=\frac{c_3m}{5}$.
	\qed

\subsection{Proof of Lemma \ref{lem:intro.atom}}

	The proof is based on a similar construction to the one used in the proof of Lemma \ref{lem:psi1.needed}.
	
	Consider  $t_0$ that satisfies  $\P(w=t_0)=\eta$ and set $H(\cdot)=\P(w\leq \cdot)$.
	Since $w$ is symmetric, we may assume without loss of generality that $t_0\geq 0$.
	
	Let $\delta\in(0,\frac{1}{10})$ and $d\in\N$ to be specified in what follows and again set 
	\[A=\left\{ \sqrt{1-\delta^2} \cdot e_1 + \delta  \cdot  e_k : 2\leq k \leq d\right\}.\]
	Moreover, put $\alpha=\frac{1}{10}$ and let $\beta=\P(w\geq \alpha)$.
	Recalling that $w$ is `nontrivial' we have that $\beta\geq \frac{1}{20}$.
	
	Let us show that there is some $\delta_0>0$ (that may depend on $H$) satisfying that for every $\delta\leq \delta_0$,
	\begin{align}
	\label{eq:atoms.cdf}
	\sup_{x\in A} F_x(t_0)\leq H(t_0)- \frac{\beta}{2}\eta.
	\end{align}
	Indeed, set $x(\delta)=\sqrt{1-\delta^2}e_1 +\delta e_2\in A$, and since $X$ has iid coordinates,  $F_y=F_{x(\delta)}$ for every $y\in A$.
	Moreover, if $w'$ is an independent copy of $w$, then
	\begin{align*}
	F_{x(\delta)}(t_0)
	&=\P\left(\sqrt{1-\delta^2}w' +\delta w \leq t_0  \right)
	=\E H\left(\frac{ t_0 -\delta w}{\sqrt{1-\delta^2}}\right) \\
	&=\E  (1_{w\leq 0 } +1_{w>0})H\left(\frac{ t_0 -\delta w}{\sqrt{1-\delta^2}}\right).
	\end{align*}
Observe that if $w \leq 0$ then $\frac{ t_0 -\delta w}{\sqrt{1-\delta^2}}$ decreases to $t_0$ as $\delta \to 0^+$, and if $w>0$ then $\frac{ t_0 -\delta w}{\sqrt{1-\delta^2}}$ increases to $t_0$ as $\delta \to 0^+$.
	By the right-continuity of $H$ and the dominated convergence theorem,
	\begin{align*}
	\lim_{\delta\to 0^+ } F_{x(\delta)}(t_0)
	&= \P(w\leq 0) H(t_0) + \P(w>0) H(t_0-)
	\end{align*}
	where $H(t_0-)=\lim_{s\uparrow t_0} H(s)= H(t_0)-\eta$.
	Clearly $\P(w>0)\geq  \P(w\geq \alpha)= \beta$ by definition and  it follows that \eqref{eq:atoms.cdf} holds for all $\delta\leq\delta_0$.

	 Next, set
	 \[ I=\{ i \in\{1,\dots,m\} : \inr{X_i,e_1}\leq t_0\},\]
	 put $\varepsilon=\frac{\beta \eta}{4 H(t_0)}$, and let  $\Omega_1(\mathbb{X})$  be the event in which $|I|\geq (1-\varepsilon) H(t_0)m$. 
	 By Markov's inequality, 
	 \[\P(\Omega_1(\mathbb{X})^c )
	 \leq  \P\big( ||I| - m H(t_0)|\geq \varepsilon H(t_0)m \big)
	 \leq \frac{ m H(t_0)(1-H(t_0))}{\varepsilon^2 H(t_0)^2 m^2}
	 \leq 0.01,\]
	 where the last inequality holds if $m\geq  \frac{c_1(\beta)}{\eta^2}$.
	Setting
	 \[\Omega_2(\mathbb{X})=\{ \text{there is } 2\leq k \leq d \text{ such that } \inr{X_i,e_k}\leq - \alpha \text{ for every } 1\leq i\leq m\},\]
	the same arguments  used in the proof of Lemma \ref{lem:psi1.needed} show that if $d=c_2 \beta^{-m}$ then
	 \[ \P(\Omega_2(\mathbb{X}))
	 = 1- (1-\beta^m)^{d-1}
	 \geq 0.99. \]
	
	 For a realization $(X_i)_{i=1}^m\in\Omega_1(\mathbb{X})\cap\Omega_2(\mathbb{X})$ there is $x\in A$ satisfying that for every $i\in I$,
	 \begin{align*}
	 \inr{X_i,x}
	& =\sqrt{1-\delta^2}\inr{X_i,e_1} + \delta\inr{X_i,e_k}
	\\
	& \leq \sqrt{1-\delta^2} t_0 -\alpha \delta
	=\sqrt{1-\delta^2} t_0 - \frac{\delta}{10}
	 \leq t_0.
	 \end{align*}
	 In particular, $F_{m,x}(t_0)\geq \frac{|I|}{m}$, and by the choice of $\varepsilon$ and  \eqref{eq:atoms.cdf},
	 \[ F_{m,x}(t_0) -F_{x}(t_0)
	 \geq  \left(1-\frac{\beta \eta}{4 H(t_0)}\right) H(t_0) - \left( H(t_0) -\frac{\beta}{2}\eta \right)
	 \geq \frac{\beta}{4}\eta. \]
	
	To complete the proof, recall that   $\gamma_1(A)\leq c_43\delta \log(d)$ and  that $d=c_2\beta^{-m}$.
	Therefore, $\gamma_1(A)\leq c_4 \delta m \log(\frac{1}{\beta})$, and for $\delta\leq  \delta_0$ small enough we have that  $\gamma_1(A)\leq 10$.
	\qed

\subsection{A variance dependent DKW inequality}
\label{sec:variance}
One may rightly ask if choosing the $L_\infty$-norm to measure the distance between $F_{m,x}$ and $F_x$ is the best option.
On the one hand, both the estimate
\[\P\left( \sup_{x\in A} \|F_{m,x} - F_x\|_{L_\infty} \geq \sqrt \Delta \right)
\leq 2\exp(-c\Delta m) \]
and the restriction $\Delta\gtrsim  \frac{\gamma_1(A)}{m}\log^2(\frac{em}{\gamma_1(A)})$ are optimal (at least up to logarithmic factors).
On the other hand, setting 
\[\sigma_x^2(t)=F_x(t)(1-F_x(t))\]
one has that $
 \V(F_{m,x}(t))=\frac{1}{m}\sigma_x^2(t)$;
therefore it is reasonable to expect that the error $|F_{m,x}(t)-F_x(t)|$ can be significantly  smaller than $\sqrt\Delta$ when $F_x(t)$ is far from  $1/2$.
And indeed, as was shown in \cite[Theorem 1.7] {bartl2023empirical}, a variant of the uniform DKW inequality with an improved \emph{variance dependent} error  is true when $X$ is the standard gaussian random vector:

\begin{theorem}[{\cite[Theorem 1.7]{bartl2023empirical}}]
\label{thm:gaussian.rato}
There are absolute constants $c_1$ and $c_2$ for which  the following holds.
Let $X=G$, the standard gaussian vector in $\R^d$. If $A \subset S^{d-1}$ is a symmetric set, $m\geq \gamma_1(A)$ and
\[
\Delta \geq
c_1 \frac{\gamma_1(A)}{m} \log^3\left(\frac{em}{\gamma_1(A)}\right),
\]
then with probability at least $1-2\exp(-c_2\Delta m)$,  for every $x\in A$ and $t\in\R$,
\begin{align}
\label{eq:thm.gaussian.varince.dependent}
| F_{m,x}(t) -  F_x(t) |
\leq  \Delta +  \sigma_x(t)\sqrt{ \Delta }.
\end{align}
\end{theorem}

The estimate in Theorem \ref{thm:gaussian.rato} is optimal, and we refer to \cite{bartl2023empirical,bartl2023variance} for more information on the problem and the proof.

It is natural to ask whether Theorem \ref{thm:gaussian.rato} extends beyond the gaussian framework.
The proof from \cite{bartl2023empirical} is `very gaussian' (as already indicated in Section \ref{sec:diff.gaussian}): it relies heavily  on several special features of the gaussian density $f_x$, in particular that in addition to being bounded, each $f_x$ satisfies that
\begin{align*}
f_x(t)\sim \sigma_x^2(t) \log\left(\frac{e}{\sigma_x^2(t) }\right).
\end{align*}

As it happens, an upper bound on the density that `scales well' with the variance is  \emph{crucial} to a variance  dependent version of the DKW inequality. As the next example shows, without that feature even `nice' random vectors (subgaussian, log-concave, with a density that is bounded by 1) do not satisfy \eqref{eq:thm.gaussian.varince.dependent}.

\begin{lemma}
\label{lem:intro.ratio.not.true}
	There are constants   $c_1,\dots, c_5$ and for every $m\geq c_1$ there is $d\in\mathbb{N}$ and $A\subset S^{d-1}$ with $\gamma_1(A)\leq c_2$ for which the following holds.
	Let $X$ be uniformly distributed in $[-\sqrt 3,\sqrt 3]^d$.
	Then there is $t_0\in\R$ that satisfies
$\sigma_x^2(t_0)\in[\frac{c_3}{\sqrt m}, \frac{c_4}{\sqrt m}]$ for every $x\in A$
 and, 	with probability at least $0.9$, 
	\[\sup_{x\in A} \left| F_{m,x}(t_0)-F_{x}(t_0)\right|\geq  \frac{c_5}{\sqrt m}.\]
\end{lemma}

If the assertion of Theorem \ref{thm:gaussian.rato} for the random vector in Lemma \ref{lem:intro.ratio.not.true} were true, then with high probability (and ignoring logarithmic factors)
\[\sup_{x\in A}| F_{m,x}(t_0)-F_{x}(t_0)|
\lesssim   \frac{1}{m} + \frac{\sigma_x(t_0)}{\sqrt m}
\sim   \frac{1}{ m^{3/4}};\]
that is far from  the true behaviour for $X$.
In particular, ignoring all logarithmic factors in what follows, consider the best constant $a\geq 0$ for which, for $\Delta\gtrsim \frac{\gamma_1(A)}{m}$,  with at least constant probability, for every $x\in A$ and $t\in\R$, 
\[| F_{m,x}(t)-F_{x}(t)|
\leq a\left( \Delta + \sqrt{\Delta}\sigma_x(t)  \right). \]
Lemma \ref{lem:intro.ratio.not.true} shows that $a\gtrsim m^{1/4}$.

The proof of Lemma \ref{lem:intro.ratio.not.true} follows the arguments used in the proofs of Lemmas \ref{lem:intro.atom} and \ref{lem:psi1.needed} but requires certain modifications.

\begin{proof}[Proof of Lemma \ref{lem:intro.ratio.not.true}]
	Let $Y=X/\sqrt 3$ and note that $Y$ is uniformly distributed in $[-1,1]^d$.
	Set  $\delta=\frac{1}{\sqrt m}$ and  $d=\exp(c_0\sqrt{m})$ for a suitable constant $c_0$ that is specified in what follows,  and put
\[ A = \left\{ \sqrt{1-\delta^2} \cdot e_1 + \delta \cdot e_k  : k=2,\dots, d\right\}
\subset S^{d-1};\]
	thus  $\gamma_1(A)\leq c_1\delta\log(d)\leq c_1c_0$.
	
	Denote by $I$ the set of indices corresponding to the smallest $\frac{1}{3}\delta m$ coordinates of $(\inr{Y_i,e_1})_{i=1}^m$ and let
	\[ \Omega_1(\mathbb{Y})=
	\left\{ \inr{Y_i,e_1} \leq -1 + \frac{7}{10} \delta  \text{ for every } i\in I \right\}. \]
	Since
	\[ \P\left(\inr{Y,e_1}  \leq  -1 + \frac{7}{10}  \delta \right)  = \frac{7}{20}\delta
	>\frac{1}{3}\delta ,\]
 	 it follows from Markov's inequality that $\P(\Omega_1(\mathbb{Y}))\geq 0.99$ when  $m\geq \frac{c_1}{\delta}= c_1\sqrt m$.

	Set
	\[ \Omega_2(\mathbb{Y})
	=\left\{ \text{there is } 2\leq k\leq d \text{ such that } \inr{Y_i,e_k} \leq -\frac{8}{10} \text{ for every }i\in I \right\}.\]
	Clearly $\P(\inr{Y,e_k}\leq -\frac{8}{10} )=\frac{1}{10}$, and using the independence of the coordinates of $Y$,
\[ \P\big (\Omega_2(\mathbb{Y}) \, \big| \, \Omega_1(\mathbb{Y}) \big)
=1 - \left( 1- 10^{-|I|}\right)^{d-1}
\geq 1-\exp\left(- (d-1) 10^{-|I|} \right)
\geq 0.99.\]
Indeed, the first inequality holds because  $\log(1-\lambda)\geq-\lambda$ for $\lambda\in(0,1)$, and the second one follows if $d\geq c_2 10^{|I|}$. Since $|I|=\frac{1}{3}\sqrt m$ and $d=\exp(c_0\sqrt m)$ for a constant $c_0$ that we are free to choose, the latter is satisfied.

Fix a realization $(Y_i)_{i=1}^m\in\Omega_1(\mathbb{Y})\cap \Omega_2(\mathbb{Y})$.
Then there exists $x= \sqrt{1-\delta^2}  e_1 + \delta e_k\in A$ and at least $\frac{1}{3}\delta m$ indices $i$ for which
\begin{align*}
 \inr{Y_i,x}
& =\sqrt{1-\delta^2} \inr{Y_i, e_1} +\delta \inr{Y_i,e_k}\\
 &\leq \sqrt{1-\delta^2} \cdot \left(-1+ \frac{7}{10} \delta \right) - \frac{8}{10}\delta
=(\ast).
 \end{align*}
Moreover, if  $\delta=\frac{1}{\sqrt m}\leq c_3$ then $(\ast)\leq -1$ and in particular
\[\P_m(\inr{Y,x}\leq -1)
\geq 
\frac{1}{3}\delta.\]
On the other hand,  it is straightforward to show that for every $x\in A$,
\[ \P(\inr{Y,x} \leq -1)\in \left[c_4\delta , \frac{1}{4}\delta\right].\]

Setting $t_0=-\sqrt{3}$ it follows that $F_{m,x}(t_0)-F_x(t_0)\geq \frac{1}{12}\delta$ and that $F_z(t_0)\in [c_4  \delta, \frac{\delta}{4}]$ for every $z\in A$, as claimed.
\end{proof}

\section{Application: Uniform estimation of monotone functions}
\label{sec:estimating.monotone.functions}

A problem one encounters frequently in high dimensional statistics is estimating the means $\E\varphi(\inr{X,x}) $
for a given function $\varphi\colon\R\to\R$, and uniformly for every $x\in A$ ---
using only a sample $X_1,\dots, X_m$ consisting of independent copies of  $X$.

For example, $\varphi(t)=t^2$ corresponds to \emph{covariance estimation}, where an intuitive (yet ultimately suboptimal) approach is to use  the empirical mean  $\frac{1}{m}\sum_{i=1}^m \inr{X_i,x}^2$ as an estimator.
With that choice the resulting error  is
\[
\sup_{x\in S^{d-1} } \left| \frac{1}{m}\sum_{i=1}^m \inr{X_i,x}^2 -\E \inr{X,x}^2 \right|
=\sup_{x\in S^{d-1} } \left| \int_{\mathbb{R}} t^2 \, d F_{m,x}(t)  -\int_{\mathbb{R}} t^2 \, d F(t) \right| .
\]
Another example is estimating  $\E |\inr{X,x}|^p$ for every $x\in A$ and $p>2$ (see, e.g., \cite{guedon2007lp}) which naturally leads to the  choice   $\varphi(t)=|t|^p$.
Just as in the case $p=2$, the empirical $p$-mean happens to be a  suboptimal procedure. 
In both cases, the suboptimality is caused by \emph{outliers}: atypically large values of $|\inr{X_i,x}|$ that occur with nontrivial probability for some $x\in S^{d-1}$ and some $1\leq i\leq m$.

As it happens, a \emph{trimmed $p$-mean} estimator can be shown to overcome the effect of outliers \cite{mendelson2021approximating}, and exhibits optimal statistical performance.
To formulate that result, for $w\in\R^m$ denote by $w^\sharp$ its monotone non-decreasing rearrangement and for $\delta\in(0,\frac{1}{10})$ set
\[ \mathcal{E}_\delta(w)
= \frac{1}{m}\sum_{i=\delta m}^{(1-\delta) m} w^\sharp_i. \]
Let $L,p\geq 2$, and $q>2p$. Set $\varphi(t)=|t|^p$, let $X$ be  an isotropic random vector in $\R^d$, and assume that $\sup_{x\in S^{d-1}} \E |\inr{X,x}|^q \leq L^q$.

\begin{theorem}[{\cite[Theorem 1.7]{mendelson2021approximating}}]
\label{thm:lp.shahar}
	There are constants $c_1,\ldots, c_4 $ that depend only on $L, p, q$ for which the following holds.
	If $m\geq d$ and
	\[ \Delta \in\left[ c_1 \frac{d}{m}\log\left(\frac{em}{d}\right) \, , \, 1 \right], \]
	then with probability at least $1-2\exp(-c_2\Delta m)$,
\begin{align}
\label{eq:truncated.mean}
\sup_{ x\in S^{d-1} } \left|
\mathcal{E}_{c_3\Delta}\Big( \big(\varphi( \inr{X_i,x} )\big)_{i=1}^m  \Big)
- \E\varphi(\inr{X,x}) \right|
\leq c_4 \sqrt\Delta.
\end{align}
\end{theorem}

Theorem \ref{thm:lp.shahar} was extended in \cite{bartl2022structure}: \eqref{eq:truncated.mean} actually holds for any $\varphi$ in the class
\[  \mathcal{M}_{p,\beta}
=\left\{ \varphi\colon\R\to\R : \varphi \text{ is monotone and } \sup_{t\in\R } \frac{|\varphi(t)|}{ 1+|t|^p } \leq \beta \right\} , \]
where $\beta>0$ and $p\geq 2$.

There are two issues with the estimator from Theorem \ref{thm:lp.shahar}.
Firstly, it deals only with the class of linear functionals  index by $S^{d-1}$  (thus forcing that $m\geq d$ and $\Delta\gtrsim \frac{d}{m}$) and not general sets $A\subset S^{d-1}$, nor with more general classes of functions $\mathcal{H}$.
Secondly, the trimmed mean  estimator is highly nonlinear.

Here, we resolve both issues by constructing a linear estimator that is (almost) optimal for an arbitrary  class of functions $\mathcal{H}$ that satisfies Assumption \ref{ass:tails}.
More accurately,  an \emph{estimator} here is a measurable function $\widehat{\nu}\colon (\R^d)^m\to\mathcal{P}(\R^d)$ that assigns to a sample $(X_i)_{i=1}^m$ a probability measure on $\R^d$.
Denote by  $\E_{\widehat{\nu}}$  the expectation with respect to the measure $\widehat{\nu}$, and in particular, for every $h\in\mathcal{H}$,
\[\E_{\widehat{\nu}} \, \varphi( h(X))=\int_{\R^d} \varphi(h(w)) \,\widehat{\nu}(dw).\]

\begin{theorem} \label{thm:l.p.moments}
	For every $L,D,p,\beta\geq 1$ there are constants $c_1,c_2,c_3$ that depend only on $L,D,p,\beta$ for which the following holds.
	Suppose that $\mathcal{H}$ satisfies Assumption \ref{ass:tails}, $\gamma_1(\mathcal{H})\geq 1$,  $m\geq \gamma_1(\mathcal{H})$ and
\[
 \Delta  \in\left[ c_1 \frac{\gamma_1(\mathcal{H})}{m} \log^2\left(\frac{em}{\gamma_1(\mathcal{H})}\right) \, , \, 1 \right],
\]
then there exists an estimator $\widehat{\nu}$, for which with probability at least $1-2\exp(-c_2\Delta m)$,
	\begin{align}
	\label{eq:estimator.monotone.functions}
	\sup_{h\in \mathcal{H}} \sup_{\varphi \in \mathcal{M}_{p,\beta}} \big| \E_{\widehat{\nu}}\,\varphi( h(X)) - \E \varphi(h(X)) \big|
	\leq c_3 \sqrt\Delta \log^p\left(\frac{e}{\Delta}\right).
	\end{align}
\end{theorem}

Note that the estimator  $\widehat{\nu}$ depends on the parameters $L,D,p,\beta$ and the wanted accuracy level $\Delta$, but not on any other feature of $X$ or of $\mathcal{H}$.

\vspace{0.5em}
The proof of  Theorem \ref{thm:l.p.moments} requires some preparations. We may assume that $\sqrt{\Delta}\leq \frac{1}{10}$, and
let $\Omega(\mathbb{X})$ be the event in which the assertion of  Theorem \ref{thm:unif.DKW.class} holds, that is,
\begin{align}
\label{eq:in.appl.dkw}
\sup_{h\in \mathcal{H}} \|F_{m,h}-F_h\|_{L_\infty}\leq \sqrt{\Delta}.
\end{align}
Hence, $\P(\Omega(\mathbb{X}))\geq 1-2\exp(-c_0 \Delta m)$ for a constant $c_0=c_0(L,D)$.

Fix a realization $(X_i)_{i=1}^m\in\Omega(\mathbb{X})$.
Clearly, \eqref{eq:in.appl.dkw} implies that for every $h\in\mathcal{H}$ and  $u\in(\sqrt{\Delta}, 1-\sqrt\Delta)$,
\begin{align}
\label{eq:inverse.shift}
	F_{m,h}^{-1}(u) \in\left[ F_h^{-1}(u-\sqrt\Delta), F_h^{-1}(u +\sqrt\Delta) \right] .
\end{align}
Moreover, by the subexponential tail decay of $h(X)$, it follows that for every $u\in(0,1)$,
\begin{align}
\label{eq:inverse.growth}
	|F_{h}^{-1}(u)|
	\leq L \log\left( \frac{e}{ \min\{u,1-u\} }\right).
\end{align}

\begin{lemma}
\label{lem:dkw.implies.monotone}
	For every $L,p,\beta\geq 1$ there is a constant $c$ that depends only on $L,p,\beta$ for which the following holds.
	For every $\varphi\in \mathcal{M}_{p,\beta}$, $h\in \mathcal{H}$, and $(X_i)_{i=1}^m\in\Omega(\mathbb{X})$,
	\[\left| \int_{\sqrt\Delta}^{1-\sqrt\Delta} \varphi( F_{m,h}^{-1}(u))\,du - \E\varphi(h(X)) \right|
	\leq c \sqrt\Delta \log^p\left(\frac{e}{\Delta}\right).\]
\end{lemma}
\begin{proof}
	Fix $\varphi\in \mathcal{M}_{p,\beta}$ and assume without loss of generality that $\varphi$ is non-decreasing.
	Define $I^+_h$ and $I_h^-$ by
	\[ I^\pm_h= \int_{\sqrt\Delta}^{1-\sqrt\Delta} \varphi\left(F_{h}^{-1}(u \pm \sqrt\Delta) \right) \,du.\]
	Invoking \eqref{eq:inverse.shift} and since $\varphi$ is non-decreasing,
	\[ \int_{\sqrt\Delta}^{1-\sqrt\Delta} \varphi(F_{m,h}^{-1}(u)) \,du
	\in \left[ I_h^- \, , \, I_h^+ \right].
	\]
	A change of variables $u\mapsto u\pm \sqrt\Delta$ shows that
	\begin{align*}
	\left| I^\pm_h
	- \int_0^{1} \varphi(F_h^{-1}(u)) \,du \right|
	&\leq    2\int_{[0,1]\setminus[2\sqrt\Delta,1-2\sqrt\Delta] } |\varphi(F_h^{-1}(u))|\,du  \\
	&\leq c(L,p,\beta) \sqrt\Delta\log^p\left(\frac{e}{\Delta}\right),
	\end{align*}
	where the last inequality follows from
	the  tail behaviour \eqref{eq:inverse.growth} and the growth assumption that  $|\varphi(\cdot)|\leq \beta(1+|\cdot|^p)$.
	Finally, by tail-integration,  $\E\varphi(h(X))=\int_0^1\varphi(F_h^{-1}(u))\,du$, from which the claim follows.
\end{proof}

The proof of Theorem \ref{thm:l.p.moments} is based on an elegant idea from \cite{abdalla2022covariance}. 
It illustrates how to construct a linear estimator from a family of non-linear estimators for its evaluations.
In our context, it shows how to construct an estimator for $\nu$ (that is, a probability measure) from the family $(\int_{\sqrt{\Delta}}^{1 - \sqrt{\Delta}} \varphi(F_{m,h}^{-1}(u))\, du)_\varphi$.
The approach is detailed below.

\begin{proof}[Proof of Theorem \ref{thm:l.p.moments}]
Let $c=c(L,p,\beta)$ be the constant appearing in Lemma \ref{lem:dkw.implies.monotone} and set $\delta=c \sqrt{\Delta} \log^p\left( \frac{e}{\Delta} \right)$.
For every $\varphi\in \mathcal{M}_{p,\beta} $ and $h\in \mathcal{H}$, put
\[ \mathcal{P}_{\varphi,h}
= \left\{ \nu\in \mathcal{P}(\mathbb{R}^d ) : \left| \E_\nu \, \varphi(h(X))  - \int_{\sqrt\Delta}^{1-\sqrt{\Delta}}  \varphi( F_{m,h}^{-1}(u)) \,du  \right|
\leq \delta \right\}
\]
(where, if $\varphi(h(X))$ is not integrable with respect to  $\nu$ we set  $\E_\nu \varphi(h(X))=\infty$).

By Lemma \ref{lem:dkw.implies.monotone}, for every realization $(X_i)_{i=1}^m\in\Omega(\mathbb{X})$, the set $\mathcal{P}_{\varphi,h}$ contains $\mu$.
In particular, on that event
\[ \bigcap\left\{  \mathcal{P}_{\varphi,h} : \varphi \in\mathcal{M}_{p,\beta}, \, h\in \mathcal{H} \right\}\neq\emptyset,\]
 and set $\widehat{\nu}$ to be any element in that intersection.
 Hence, for any  $\varphi\in \mathcal{M}_{p,\beta}$ and $h\in \mathcal{H}$, it follows from the definition of $\mathcal{P}_{\varphi,h}$ and  Lemma \ref{lem:dkw.implies.monotone} that on $\Omega(\mathbb{X})$,
\begin{align*}
\left| \E_{\widehat{\nu}}\, \varphi(h(X))
- \E \varphi(h(X)) \right|
&\leq
 \left| \E_{\widehat{\nu}}\, \varphi( h(X))
 - \int_{\sqrt\Delta}^{1-\sqrt{\Delta}}  \varphi( F_{m,h}^{-1}(u)) \,d u \right| \\
&\qquad\qquad + \left| \int_{\sqrt\Delta}^{1-\sqrt{\Delta}}  \varphi( F_{m,h}^{-1}(u)) \,d u - \E \varphi(h(X)) \right|
\leq 2\delta,
\end{align*}
as claimed.
\end{proof}

\subsection{Uniform estimates on Wasserstein distances}

As was shown in Theorem \ref{thm:l.p.moments}, the uniform DKW inequality can be used to construct an estimator $\widehat{\nu}$ of the underlying distribution $\mu$ of $X$ that is `close'   in the sense of integration relative to a (large) family of test functions.
An alternative and more standard approach of comparing similarity of measures is via the Wasserstein distance.

\begin{definition}
For two probability measures $\alpha$ and $\beta$ on $\R$ and $p\geq1$,  the $p$-order \emph{Wasserstein distance}  is
\[ \mathcal{W}_p(\alpha,\beta)
=\left( \inf \int_{\R\times\R} |s-t|^p \,\pi(ds,dt) \right)^{1/p} ,\]
with the infimum  taken over all couplings $\pi$, i.e.\ over all probability measures whose first marginal is $\alpha$ and  the second marginal is $\beta$.
\end{definition}

We refer to  \cite{figalli2021invitation,villani2021topics} for more information on the $\mathcal{W}_p$ distance and optimal transport.
\vspace{0.5em}

An estimate on $\sup_{x \in A} \mathcal{W}_2(F_{m,x}, F_x)$ for general $A\subset S^{d-1}$ was recently established in  \cite{bartl2023empirical} for the  standard gaussian vector.
There are other estimates on $\sup_{x \in A} \mathcal{W}_p(F_{m,x}, F_x)$ for more general random vectors but only for $A=S^{d-1}$ --- in which case the distance is the so-called \emph{max-sliced Wasserstein distance} \cite{deshpande2019max,lin2021projection},  see
 \cite{bartl2022structure,manole2022minimax,nietert2022statistical,olea2022generalization}
and the references therein.

The methods developed in the previous sections can be used for estimating \linebreak $\sup_{h\in \mathcal{H}} \mathcal{W}_1(F_{m,h}, F_h)$ for an arbitrary  class of functions $\mathcal{H}$ (and in particular for a class of linear functionals index by a general set $A\subset S^{d-1}$ and a relatively general random vector):

\begin{theorem}
\label{thm:wasserstein}
	For every $L,D\geq 1$ there are constants $c_1,c_2,c_3$ that depend only on $L$ and $D$ for which the following holds.
	Suppose that $\mathcal{H}$ satisfies Assumption \ref{ass:tails}, $\gamma_1(\mathcal{H})\geq 1$, $m\geq \gamma_1(\mathcal{H})$ and
\[
	\Delta\in\left[
c_1 \frac{\gamma_1(\mathcal{H})}{m} \log^2\left(\frac{em}{\gamma_1(\mathcal{H})}\right) \, , \, 1 \right],
\]
 then with probability at least $1-2\exp(-c_2\Delta m)$,
	\[ \sup_{h\in \mathcal{H}} \mathcal{W}_1\left( F_{m,h}, F_h \right)
	\leq c_3   \sqrt{\Delta } \log\left(\frac{e}{\Delta}\right).\]
\end{theorem}

\begin{proof}[Sketch of proof] 
	Let $\mathcal{L}_1$ be the set of all real $1$-Lipschitz functions $\varphi$ that satisfy $\varphi(0)=0$.
	By the Kantorovich-Rubinstein duality (see, e.g., \cite[Proposition 2.6.6]{figalli2021invitation}),
	\begin{align}
	\label{eq:wasserstein.1.duality}
	\mathcal{W}_1\left( F_{m,h}, F_h \right)
	&= \sup_{\varphi\in\mathcal{L}_1} \left( \int_\mathbb{R} \varphi(t)\, dF_{m,h}(t) -  \int_\mathbb{R} \varphi(t)\, dF_{h}(t)\right)
	 \\
	\nonumber
	&= \sup_{\varphi \in \mathcal{L}_1 } \left( \int_0^1\varphi( F_{m,h}^{-1}(u))\, du  -  \int_0^1\varphi( F_{h}^{-1}(u))\, du \right).
	\end{align}
	Every $\varphi\in \mathcal{L}_1$ can be written as the difference of two increasing functions with sublinear growth, and it follows from Lemma \ref{lem:dkw.implies.monotone}  that with probability at least $1-2\exp(c_1\Delta m)$, for every $\varphi\in\mathcal{L}_1$,
	\[\left| \int_{\sqrt\Delta}^{1-\sqrt\Delta}\varphi( F_{m,h}^{-1}(u))\, du  -  \int_{\sqrt\Delta}^{1-\sqrt\Delta}\varphi( F_{h}^{-1}(u)) \right|
	\leq c_2 (L,D) \sqrt\Delta \log\left(\frac{e}{\Delta}\right). \]
	All that remains is to control the `tails', that is, estimate the integrals in the intervals $[0,\sqrt\Delta]$ and $[1-\sqrt\Delta,1]$.
	
	Set $U=[0,1]\setminus [\sqrt\Delta,1-\sqrt\Delta]$ and recall that by the tail behaviour of $F_h$ (see \eqref{eq:inverse.growth}),
	\[\sup_{\varphi\in\mathcal{L}_1} \int_U\varphi( F_{h}^{-1}(u))\, du
	\leq \int_U | F_{h}^{-1}(u)|\, du
	\leq c_3(L)\sqrt\Delta\log\left( \frac{e}{\Delta} \right).\]
	Moreover,  using the notation of Section \ref{sec:large.coord},
	\[ \sup_{\varphi\in\mathcal{L}_1}\int_U\varphi( F_{m,h}^{-1}(u))\, du
	\leq \int_U | F_{m,h}^{-1}(u)|\, du
	\sim \frac{1}{m}\sum_{i=1}^{2\sqrt\Delta m} h(X_i)^\ast
	=\Psi(x).\]
	A standard chaining argument similar to the one used in Section \ref{sec:large.coord} shows that with probability at least $1-2\exp(-c_4\Delta m)$,
	\[\sup_{h\in \mathcal{H}} \Psi(x)\leq c_5(L)\sqrt\Delta\log\left( \frac{e}{\Delta} \right),\]
	and the proof is completed by combining these observations.
\end{proof}

\begin{remark}
	By applying the uniform DKW inequality it is possible to extend the scope of Theorem \ref{thm:wasserstein} and analyse higher order Wasserstein distances $\mathcal{W}_p$ as well (rather than  just $\mathcal{W}_1$).
	However, the proof is more involved because a representation as in  \eqref{eq:wasserstein.1.duality} is no longer true when $p>1$.
	We defer those results to future work.
\end{remark}

\appendix
\section{A Sudakov-type bound}
\label{sec:Sudakov}

Theorem \ref{thm:unif.DKW.class} (applied to the class of linear functions) has a restriction on the choice of $\Delta$: that
\[
\Delta\gtrsim \frac{\gamma_1(A)}{m}\log^2\left(\frac{em}{\gamma_1(A)}\right).
\]
At a first glance this restriction appears strange.
As we explain here, it is (essentially) optimal and the  argument is based on a Sudakov-type lower bound. 
Before formulating that lower bound, let us outline the standard relations between the $\gamma_1$-functional and \emph{covering numbers}.

Let $\mathcal{N}(A,\delta B_2^d)$  be the smallest number of (open) Euclidean balls of radius $\delta$ needed to cover $A\subset \R^d$.
It is standard to verify that for any $A \subset \R^d$,
\begin{align}
\label{eq:gamma1.entropy}
c_1 \sup_{\delta \geq 0 } \left( \delta \cdot  \log \mathcal{N}\left(A,\delta B_2^d\right) \right)
\leq
\gamma_1(A)
\leq
c_2 \int_0^\infty  \log \mathcal{N}\left(A,\delta B_2^d\right)  \,d\delta
\end{align}
for absolute constants $c_1$ and $c_2$.
Indeed, the lower bound holds because for an (almost) optimal admissible sequence and every $s$,
$$
\sup_{x \in A} 2^s \|x - \pi_s x\|_2 \leq2 \gamma_1(A);
$$
the upper bound follows by using suitable $\delta$-covers of $A$ as the sets $A_s$.

When $A\subset S^{d-1}$, the inequalities in \eqref{eq:gamma1.entropy} are sharp up to a logarithmic factor in the dimension $d$; in particular
\begin{align}
\label{eq:gamma1.entropy.dimension}
\gamma_1(A)
\leq c\log(ed) \cdot  \sup_{\delta \geq 0 } \left(\delta \cdot  \log \mathcal{N}\left(A,\delta B_2^d\right) \right),
\end{align}
see e.g., \cite[Section 2.5]{talagrand2022upper}.


\begin{proposition}[{\cite[Proposition 1.9]{bartl2023empirical}}]
 \label{prop:intro.sudakov}
	There are absolute constants $c_1$ and $c_2$ such that the following holds.
	Let $X=G$   be the standard gaussian random vector, let  $A \subset S^{d-1}$ be  symmetric and put $\delta>0$.
	If
	\[m\geq c_1\max\left\{ \frac{1}{\delta^2} , \frac{ \log \mathcal{N}(A,\delta B_2^d)}{\delta}\right\},\]
	  then with probability at least $0.9$,
	\begin{align}	
	\label{eq:intro.sudakov}
	 \sup_{x\in A} \| F_{m,x} - F_x \|_{L_\infty}
	\geq c_2 \sqrt \frac{ \delta \cdot  \log \mathcal{N}\left(A,\delta B_2^d\right)  }{m} .
	\end{align}
\end{proposition}

Following \eqref{eq:gamma1.entropy} and \eqref{eq:gamma1.entropy.dimension}, the right hand side in \eqref{eq:intro.sudakov} almost coincides with $\sqrt{ \gamma_1(A)/m }$.
In particular, a straightforward consequence of  Proposition \ref{prop:intro.sudakov} is that asymptotically, $\gamma_1$ is the correct notion of complexity (up to logarithmic factors):
for \emph{every} set $A\subset S^{d-1}$ there is $m_0(A)$ and for all $m\geq m_0(A)$, with probability at least $0.9$,
	\begin{align*}
	 \sup_{x\in A}\| F_{m,x} - F_x \|_{L_\infty}
	 \geq c \sqrt{ \frac{ \gamma_1(A)}{ m }}   \frac{ 1}{  \sqrt{\log(d)}}.
	\end{align*}
	Especially, 
	\[\liminf_{m\to\infty}   \E \left( \sup_{x\in A}  \sqrt m \| F_{m,x} - F_x \|_{L_\infty} \right)
	 \geq c' \sqrt{  \frac{\gamma_1(A)}{ \log(d)}}.\]	
\begin{remark}\hfill
\begin{enumerate}[(a)]
\item
	The proof of Proposition \ref{prop:intro.sudakov}  extends (without any modifications) to the more general setting considered in this article.
	Indeed, the key feature of the gaussian random vector used in the proof is a certain small ball assumption.
	In the general case, if  $\mathcal{H}$ is an arbitrary class of functions, that assumption corresponds to
	\begin{align}
	\label{eq:for.sudakov}
	 \P\left( \{ h(X) \leq 0 \} \,\bigtriangleup \,  \{ h'(X)\leq 0 \} \right)
	\geq  \kappa \|h-h'\|_{L_2},
	\end{align}
	and in that case the proof is identical to that of Proposition \ref{prop:intro.sudakov}, with the constants $c_1$ and $c_2$ depending on the small ball constant $\kappa$.
	\item
	Note that if $\mathcal{H}=\{\inr{\cdot,x}: x\in A\}$ for $A\subset S^{d-1}$ and $X$ is rotation invariant, then \eqref{eq:for.sudakov} holds for an absolute constant $\kappa$.
	\end{enumerate}
\end{remark}

\vspace{1em}
\noindent
\textsc{Acknowledgements:}
Daniel Bartl is grateful for financial support through the Austrian Science Fund (FWF) [doi: 10.55776/P34743 and 10.55776/ESP31], the Austrian National Bank [Jubil\"aumsfond, project 18983], and a Presidential-Young-Professorship grant [`Robust Statistical Learning from Complex Data'].


\begin{thebibliography}{10}

\bibitem{abdalla2022covariance}
P.~Abdalla and N.~Zhivotovskiy.
\newblock Covariance estimation: Optimal dimension-free guarantees for
  adversarial corruption and heavy tails.
\newblock {\em Journal of the European Mathematical Society, to appear}, 2022.

\bibitem{adamczak2011restricted}
R.~Adamczak, A.~E.~Litvak, A.~Pajor, and N.~Tomczak-Jaegermann.
\newblock Restricted isometry property of matrices with independent columns and
  neighborly polytopes by random sampling.
\newblock {\em Constructive Approximation}, 34:61--88, 2011.

\bibitem{alexander1987central}
K.~S.~Alexander.
\newblock The central limit theorem for weighted empirical processes indexed by
  sets.
\newblock {\em Journal of multivariate analysis}, 22(2):313--339, 1987.

\bibitem{alexander1987rates}
K.~S.~Alexander.
\newblock Rates of growth and sample moduli for weighted empirical processes
  indexed by sets.
\newblock {\em Probability Theory and Related Fields}, 75(3):379--423, 1987.

\bibitem{artstein2015asymptotic}
S.~Artstein-Avidan, A.~Giannopoulos, and V.~D.~Milman.
\newblock {\em Asymptotic Geometric Analysis, Part I}, volume 202.
\newblock American Mathematical Society, 2015.

\bibitem{bartl2022structure}
D.~Bartl and S.~Mendelson.
\newblock Structure preservation via the {W}asserstein distance.
\newblock {\em Journal of Functional Analysis}, 288(7):110810, 2025.

\bibitem{bartl2023empirical}
D.~Bartl and S.~Mendelson.
\newblock Empirical approximation of the gaussian distribution in
  $\mathbb{R}^d$.
\newblock {\em Advances in Mathematics},  460:110041, 2025.

\bibitem{bartl2023variance}
D.~Bartl and S.~Mendelson.
\newblock On a variance dependent {D}voretzky-{K}iefer-{W}olfowitz inequality.
\newblock {\em arXiv preprint arXiv:2308.04757}, 2023.

\bibitem{bobkov2015concentration}
S.~G. Bobkov and G.~P. Chistyakov.
\newblock On concentration functions of random variables.
\newblock {\em Journal of Theoretical Probability}, 28(3):976--988, 2015.

\bibitem{boucheron2013concentration}
S.~Boucheron, G.~Lugosi, and P.~Massart.
\newblock {\em Concentration inequalities: A nonasymptotic theory of
  independence}.
\newblock Oxford university press, 2013.

\bibitem{cantelli1933sulla}
F.~P. Cantelli.
\newblock Sulla determinazione empirica delle leggi di probabilita.
\newblock {\em Giorn. Ist. Ital. Attuari}, 4(421-424), 1933.

\bibitem{deshpande2019max}
I.~Deshpande, Y.-T. Hu, R.~Sun, A.~Pyrros, N.~Siddiqui, S.~Koyejo, Z.~Zhao,
  D.~Forsyth, and A.~G. Schwing.
\newblock Max-sliced {W}asserstein distance and its use for gans.
\newblock In {\em Proceedings of the IEEE/CVF Conference on Computer Vision and
  Pattern Recognition}, pages 10648--10656, 2019.

\bibitem{dvoretzky1956asymptotic}
A.~Dvoretzky, J.~Kiefer, and J.~Wolfowitz.
\newblock Asymptotic minimax character of the sample distribution function and
  of the classical multinomial estimator.
\newblock {\em The Annals of Mathematical Statistics}, pages 642--669, 1956.

\bibitem{figalli2021invitation}
A.~Figalli and F.~Glaudo.
\newblock {\em An {I}nvitation to {O}ptimal {T}ransport, {W}asserstein
  {D}istances, and {G}radient {F}lows}.
\newblock EMS Textbooks in Mathematics, 2021.

\bibitem{gine2006concentration}
E.~Gin{\'e} and V.~Koltchinskii.
\newblock Concentration inequalities and asymptotic results for ratio type
  empirical processes.
\newblock {\em The Annals of Probability}, 34:1143--1216, 2006.

\bibitem{gine2003ratio}
E.~Gin{\'e}, V.~Koltchinskii, and J.~A. Wellner.
\newblock Ratio limit theorems for empirical processes.
\newblock In {\em Stochastic inequalities and applications}, pages 249--278.
  Springer, 2003.

\bibitem{glivenko1933sulla}
V.~Glivenko.
\newblock Sulla determinazione empirica delle leggi di probabilita.
\newblock {\em Gion. Ist. Ital. Attauri.}, 4:92--99, 1933.

\bibitem{guedon2007lp}
O.~Gu{\'e}don and M.~Rudelson.
\newblock Lp-moments of random vectors via majorizing measures.
\newblock {\em Advances in Mathematics}, 208(2):798--823, 2007.

\bibitem{koltchinskii2003bounds}
V.~Koltchinskii.
\newblock Bounds on margin distributions in learning problems.
\newblock {\em Annales de l'IHP Probabilit{\'e}s et statistiques},
  39(6):943--978, 2003.

\bibitem{koltchinskii2002empirical}
V.~Koltchinskii and D.~Panchenko.
\newblock Empirical margin distributions and bounding the generalization error
  of combined classifiers.
\newblock {\em The Annals of Statistics}, 30(1):1--50, 2002.

\bibitem{lin2021projection}
T.~Lin, Z.~Zheng, E.~Chen, M.~Cuturi, and M.~I.~Jordan.
\newblock On projection robust optimal transport: Sample complexity and model
  misspecification.
\newblock In {\em International Conference on Artificial Intelligence and
  Statistics}, pages 262--270. PMLR, 2021.

\bibitem{lugosi2020multivariate}
G.~Lugosi and S.~Mendelson.
\newblock Multivariate mean estimation with direction-dependent accuracy.
\newblock {\em Journal of the European Mathematical Society, to appear}, 2020.

\bibitem{manole2022minimax}
T.~Manole, S.~Balakrishnan, and L.~Wasserman.
\newblock Minimax confidence intervals for the sliced wasserstein distance.
\newblock {\em Electronic Journal of Statistics}, 16(1):2252--2345, 2022.

\bibitem{massart1990tight}
P.~Massart.
\newblock The tight constant in the {D}voretzky-{K}iefer-{W}olfowitz
  inequality.
\newblock {\em The {A}nnals of Probability}, pages 1269--1283, 1990.

\bibitem{mendelson2021approximating}
S.~Mendelson.
\newblock Approximating {L}p unit balls via random sampling.
\newblock {\em Advances in Mathematics}, 386:107829, 2021.

\bibitem{nietert2022statistical}
S.~Nietert, Z.~Goldfeld, R.~Sadhu, and K.~Kato.
\newblock Statistical, robustness, and computational guarantees for sliced
  wasserstein distances.
\newblock {\em Advances in Neural Information Processing Systems},
  35:28179--28193, 2022.

\bibitem{olea2022generalization}
J.~Olea, C.~Rush, A.~Velez, and J.~Wiesel.
\newblock On the generalization error of norm penalty linear regression models.
\newblock {\em arXiv preprint arXiv:2211.07608}, 2022.

\bibitem{talagrand1996majorizing}
M.~Talagrand.
\newblock Majorizing measures: the generic chaining.
\newblock {\em The Annals of Probability}, 24(3):1049--1103, 1996.

\bibitem{talagrand2022upper}
M.~Talagrand.
\newblock {\em Upper and lower bounds for stochastic processes: decomposition
  theorems}, volume~60.
\newblock Springer Nature, 2022.

\bibitem{vaart1996weak}
A.~W.~Vaart and J.~A.~Wellner.
\newblock Weak convergence.
\newblock In {\em Weak convergence and empirical processes}, pages 16--28.
  Springer, 1996.

\bibitem{vershynin2018high}
R.~Vershynin.
\newblock {\em High-dimensional probability: An introduction with applications
  in data science}, volume~47.
\newblock Cambridge university press, 2018.

\bibitem{villani2021topics}
C.~Villani.
\newblock {\em Topics in optimal transportation}, volume~58.
\newblock American Mathematical Soc., 2021.

\end{thebibliography}

\end{document}